\documentclass[A4,10pt]{article}
\usepackage[margin=1in,dvips]{geometry}

\pdfoutput=1 
\usepackage[dvipdfmx]{graphicx}

\usepackage{color, soul}
\usepackage{amsmath, amsthm, slashed}
\usepackage{amssymb}

\usepackage{marvosym}
\usepackage{rotating}
\usepackage{mathrsfs}
\usepackage{upgreek}

\usepackage{verbatim}
\usepackage[affil-it]{authblk}
\usepackage{rotating}
\usepackage{accents}
\usepackage{harmony, slashed}
\usepackage{tikz}
\usepackage{dsfont}
\usepackage{makeidx}
\usepackage{bbm}

\usepackage{tocloft}

\usepackage{thmtools}
\usepackage{thm-restate}

\usepackage{comment}

\usepackage{enumerate}

\usepackage[colorlinks=true,linkcolor=black,citecolor=blue]{hyperref}

\usepackage{cleveref}

\newtheorem{theorem}{Theorem}[section]

\newtheorem*{conjecture*}{Conjecture}
\newtheorem*{theorem*}{Theorem}

\newtheorem{proposition}{Proposition}[subsection]
\newtheorem{definition}[proposition]{Definition}
\newtheorem{corollary}[proposition]{Corollary}
\newtheorem*{corollary*}{Corollary}
\newtheorem{lemma}[proposition]{Lemma}
\newtheorem{remark}[proposition]{Remark}

\numberwithin{equation}{section}

\newcommand{\Lie}{{\mathcal{L}}}

\newcommand{\der}{\nabla}

\newcommand{\les}{\lesssim}
\newcommand{\bea}{\begin{eqnarray}}
\newcommand{\eea}{\end{eqnarray}}

\newcommand{\derm}{ { \der^{(\bf{m})}} }
\newcommand{\rderm}{   {\mbox{$\nabla \mkern-13mu /$\,}^{(\bf{m})} }   }

\newcommand{\rpa}{\mbox{${\pa} \mkern-09mu /$}}

\newcommand{\R}{{\mathbb R}}

\newcommand{\la}{\langle}\renewcommand{\b}{\beta}

\def\a{\alpha}\def\ga{\gamma}\def\de{\delta}
\def\bm{\left( \begin{array}{cc}}
\def\endm{\end{array}\right)}\newcommand{\eq}{\end{equation}}
\def\a{\alpha}\def\b{\beta}
\def\ga{\gamma}\def\de{\delta}\def\pa{\partial}

\def \rectangle#1#2{\hbox{\vrule\vbox to #2 {\hrule\hbox to #1{\hfil}\vfil\hrule}\vrule}}

\def\a{\alpha}\def\b{\beta}\def\ga{\gamma}
\def\de{\delta}\def\pa{\partial}

\def\pa{\partial}

\def\beaa{\begin{eqnarray*}}
\def\eeaa{\end{eqnarray*}}
\def\pa{\partial}
\def\a{{\alpha}}
\def\b{{\beta}}
\def\ga{\gamma}

\def\de{\delta}

\def\la{\lambda}

\def\Om{\Omega}

\def\g{{\bf g}}

\def\SSS{{\Bbb S}}

\def\R{{\mathbb R}}

\def\12{\frac{1}{2}}

\def\bep{\begin{proposition}}
\def\eep{\end{proposition}}

\def\4{\frac{1}{4}}

\def\12{\frac{1}{2}}

\def\bep{\begin{proposition}}
\def\eep{\end{proposition}}

\def\bm#1{\boldsymbol{#1}}

\def\build#1_#2^#3{\mathrel{\mathop{\kern 0pt#1}\limits_{#2}^{#3}}}
\def\4{\frac{1}{4}}

\def\<{\langle}
\def\>{\rangle}

\DeclareMathAlphabet\mathbfcal{OMS}{cmsy}{b}{n}

\makeatletter \@addtoreset{equation}{section}  \makeatother

\title{Decoupled energy estimates for tensorial non-linear wave equations\\ and applications}

\author{Sari Ghanem}
\date{}

\begin{document}

\maketitle

\begin{abstract}

We prove energy estimates for solutions to a \textit{tensorial} system of coupled non-linear wave equations, in a way that is suitable to deal with the structure of the \textit{non-linearity} that arises from the Einstein-Yang-Mills system in the Lorenz gauge as well as with other \textit{new different} non-linearities.  We establish suitable bounds on the $L^2$-norm of \textit{each component in a frame decomposition} of the \textit{tensorial} solutions, in way that does \textit{not} involve all the other components of the tensor, which would allow us to \textit{decouple} the higher order energy estimates for certain components from the other components. We achieve this partly by exploiting the tensorial structure of the coupled non-linear wave equations, where the background metric that is \`a priori unknown, is a perturbation of the Minkowski space-time in a certain fixed system of coordinates, and by exploiting the structure of the \textit{commutator term} for the Lie derivatives of the solutions. These decoupled  energy estimates for each component of the tensor in a frame, are new and motivated by a problem that we address in a subsequent paper to prove the exterior non-linear stability of the $(1+3)$-Minkowski space-time governed by a general class of perturbations, that includes the non-linearities that arise from the Einstein-Yang-Mills system in the \textit{Lorenz gauge} as well as other \textit{new} non-linearities, which have a \textit{different} non-linear structure than the one treated by Lindblad-Rodnianski, for which their seminal $L^\infty$-estimate does \textit{not} work to the best of our knowledge. The \textit{decoupled energy bounds on each component in a frame} derived here allow us to replace the celebrated $L^\infty$-estimate of Lindblad-Rodnianski in a novel way that permits us to treat these new non-linear structures.

\end{abstract}

\setcounter{page}{1}
\pagenumbering{arabic}

\section{Introduction}\label{Introduction}\

We prove energy estimates for \textit{each component in a frame decomposition} of the \textit{tensorial} solutions to a coupled system of non-linear wave equations, \textit{without} involving all the other components. These \textit{energy estimates} on the $L^2$-norm of the higher order energy are made and tailored to prove the non-linear stability of the $(1+3)$-Minkowski space-time governed by a general class of non-linear perturbations, that includes the non-linearities that arise from the fully coupled non-linear Einstein-Yang-Mills system in the \textit{Lorenz gauge} and in wave coordinates, as well as other \textit{new different} non-linearities, which is a system of non-linear wave equations that does \textit{not} satisfy the null condition of Christodoulou, \cite{Chr1}, and of Klainerman, \cite{Kl2}, and has novel non-linearities that are \textit{different} than the one treated by Lindblad and Rodnianski, \cite{LR10}, that pose their own challenges and serious complications.

In particular, to the best of our knowledge, the $L^\infty$-estimate derived by Lindblad-Rodnianski and used in their seminal work, \cite{LR10}, does \textit{not} work for the Einstein equations coupled to the non-linear Yang-Mills fields in the \textit{Lorenz gauge}. We derive \textit{energy estimates on each component} of the tensorial solutions, \textit{decoupled from the other components}, to replace the $L^\infty$-estimate of Lindblad-Rodnianski in a way that allows us to include new non-linearities such as the one that arises from the Yang-Mills fields in the Lorenz gauge, as well as other different non-linearities. 

In this paper, we overcome the difficulty regarding the \textit{decoupling} of the \textit{higher order energy} estimates, by proving suitable bounds on \textit{each component in a frame} decomposition of the \textit{tensorial} fields solutions to the dynamically coupled system of non-linear wave equations. This is partly achieved by establishing a \textit{new} estimate on the \textit{commutator term} and by establishing an approach that is totally \textit{covariant} where we exploit the tensorial structure, despite the fact that we fix the system of coordinates. We obtain \textit{decoupled} bounds on the weighted $L^2$ norm of \textit{each component in a frame}, of the derivative of the dynamically perturbed space-time and of the \textit{tensorial} fields that are solutions to the coupled system of non-linear wave equations, and we also control a space-time integral of the tangential derivatives of \textit{each component} of the tensorial solutions in a frame decomposition.

Energy estimates on curved backgrounds form a fundamental tool to prove dispersive estimates for non-linear wave equations, see for example \cite{Alin2}, \cite{AB15}, \cite{AAG}, \cite{AAG2}, \cite{Chr1}, \cite{DR}, \cite{DR1}, \cite{DR2}, \cite{DHR}, \cite{H2}, \cite{H3}, \cite{Keir1}, \cite{Keir2} and \cite{Kl1}. These have been also used to prove decay for the Maxwell fields, see for example \cite{Alin2}, \cite{AB15_02}, \cite{BCY}, \cite{G2}, \cite{Loiz1}, \cite{Ma}, \cite{MTT}, \cite{Pas1}, \cite{Speck}, \cite{ST15}, \cite{Z} and \cite{Z2}.

However, in the non-abelian case of the Yang-Mills equations, although energy estimates were used to prove global existence of solutions, see for example \cite{CB-Chri}, \cite{CS}, \cite{EM1}, \cite{EM2}, \cite{G1} and \cite{Tao}, to the best of our knowledge, energy estimates were not successfully carried out to fully prove decay for the Yang-Mills fields on \textit{curved} backgrounds, except in \cite{G3} and thereafter in \cite{G6}-\cite{G7}. There is also a result of decay on curved background in \cite{MY1}-\cite{MY3}, however without estimates on the Yang-Mills \textit{potential}. In \cite{G4}, there is a result in higher dimensions and in \cite{G5}, there are \textit{separate energy estimates} for three space dimensions. Concerning a result of decay on the \textit{flat} Minkowski space-time, see \cite{WYY}. In \cite{AthMonYau}, there is a large data result for formation of trapped surfaces.

The difficulty comes from the fact that the non-linearity that arises from the Yang-Mills structure gives rise to stationary solutions, which do not decay, which need to be eliminated in any argument of proof of decay for the Yang-Mills fields, see the discussion in \cite{G2}. See also the important discussions in  \cite{Kadar}, \cite{Keir1}, \cite{Keir2} and \cite{LOY}, about the trouble that can be caused by non-linear structures for quasilinear wave equations.

It is precisely this difficulty that also makes the problem of stability of Minkowski space-time for the Einstein-Yang-Mills equations difficult and non-studied. Global stability of the Minkowski space-time in vacuum was teated in \cite{Bieri2}, \cite{Bieri3}, \cite{Bieri}, \cite{C-K}, \cite{F}, \cite{Graf}, \cite{Hintz}, \cite{HV}, \cite{Hu1}, \cite{LR10}, \cite{Shen1}, \cite{Shen2}, \cite{Shen3}, see also \cite{CB1}, \cite{DR2}, \cite{LR2}. Although these global stability theorems were extended to include linear matter such as the Maxwell fields -- see \cite{BCY}, \cite{Loiz1}, \cite{Speck}, \cite{Z}, \cite{Z2} --, they were not carried out to the non-linear case of the Yang-Mills fields.

In the seminal work of Lindblad-Rodnianski \cite{LR10}, they proved the non-linear stability of Minkowski space-time in wave coordinates. Their celebrated work was new and proved a result that was not expected to hold true in the harmonic gauge, by exploiting the weak-null structure of the equations in \cite{LR1}. Their original approach was based on an $L^\infty$-estimate. However, as far as our knowledge is concerned, the celebrated $L^\infty$-estimate derived by Lindblad-Rodnianski does \textit{not} work for the Einstein equations coupled to the non-linear Yang-Mills fields in the \textit{Lorenz gauge}. Hence, in order to be able to include new non-linearities such as the one that arises from the \textit{non-linear} Yang-Mills fields in the \textit{Lorenz gauge}, it is important to provide a \textit{new alternative approach} that replaces the $L^\infty$-estimate of Lindblad-Rodnianski. The difficulty of decoupling the energy estimates for each component was encountered by Lindblad-Rodnianski in \cite{LR2}-\cite{LR10} and also by Lindblad-Tohaneanu in \cite{Lind-Toh} and was \textit{not} resolved there.

In this paper, we resolve the problem by deriving the desired \textit{decoupled} energy estimates that allows us to replace their $L^\infty$-estimate, used in \cite{LR10}, in a way that allows us to include new non-linearities, such as the Yang-Mills fields in the \textit{Lorenz gauge}, for which their $L^\infty$-estimate does \textit{not} work to the best of our knowledge. In other words, we establish suitable bounds on \textit{each component in a frame decomposition} of the Lie derivatives of the tensorial solutions of the coupled system of non-linear wave equations, without involving the other components, up to some good factor that enters infront of the full components. We achieve this \textit{decoupling} of the energy estimates for each component partly by establishing a new suitable estimate on the \textit{commutator term} of the higher order energy norm.

The novelty here is that we give a new idea and different approach to treat the non-linear structure of \textit{tensorial quasilinear} wave equations by \textit{decoupling} the energy estimates for each component of the tensorial solutions in a frame decomposition. We provide suitable energy estimates that only involve \textit{each component in a frame}. This allows us to replace $L^\infty$-estimate of Lindblad and Rodnianski in \cite{LR10}, in a way that allows us to include \textit{new} non-linear structures for which their $L^\infty$-estimate does not work, to the best of our knowledge. In \cite{G6}-\cite{G7}, we use the decoupling of the energy estimates to prove the non-linear stability of $(1+3)$-Minkowski governed by the fully non-lineally coupled Einstein-Yang-Mills system in the Lorenz gauge. A key point in our approach, is that despite the fact that we work in a fixed system of coordinates, that will ultimately be chosen to be wave coordinates, we build a framework that is totally covariant and we exploit the tensorial structure of the commutator term, where some components can be decoupled from others up to some good factor.

More precisely, we consider that we are given a fixed system of coordinates, namely $(t, x^1, x^2, x^3)$\;, that is not necessarily wave coordinates, yet for our application to the proof of stability of Minkowski, this system will ultimately be chosen to be that of wave coordinates. In this fixed system of coordinates, we define $m$ to be the Minkowski metric $(-1, +1, +1, +1)$ and we define $\derm$ to be the covariant derivative associated to the metric $m$ (see Definition \eqref{definitionofMinkowskiandcovariantderivtaiveofMinkwoforafixedgivensystemofcoordinates}). We consider any arbitrary curved space-time $({\cal M}, \g) $\;, \`a priori unknown, with a smooth Lorentzian metric $\g$\;, which will ultimately be, in our application to the non-linear exterior stability problem, our unknown Lorentzian manifold solution to the fully coupled Einstein-Yang-Mills system.

We study the following system of non-linear covariant tensorial wave equations for $\Phi$ on the \`a priori unknown $({\cal M}, \g) $\;, where the initial data for the hyperbolic Cauchy problem is given on an initial Cauchy hypersurface $\Sigma$\;, and where $U\,, V$ are any vectors,
\bea\label{nonlinearsystemoftensorialwaveequations}
\begin{cases} g^{\a\b} \derm_{\a } \derm_\b  \Phi_{V} = S^{1}_{V} \; ,\\
g^{\a\b} \derm_{\a } \derm_\b  g_{UV} = S^{2}_{UV} \; .\end{cases}
\eea
Here $\Phi$ is chosen for simplicity of illustration to be a one-tensor but our argument works for any tensor of arbitrary order. In our generic energy estimates derived in this paper, the source terms $S^1$ and $S^2$ are arbitrary tensors that present some non-linearity. However, in our application to the non-linear stability of Einstein-Yang-Mills in the Lorenz gauge, they will ultimately be chosen to be the non-linearities that arise from the Einstein-Yang-Mills equations, and they will depend on $\Phi$\,, $\g$\,, $\derm \Phi$\,, $\derm \g$\,. Some components of the tensors, depending on choices for the vectors $U,\, V$ will satisfy a better tensorial wave equation with better non-linearity. Our goal is to decouple the energy estimates for these components, so that we can establish good bounds on the components that satisfy a good wave equation and use these bounds to control the full system.

In particular, to see the structure of the non-linearity arising from the Yang-Mills fields in the Lorenz gauge, we construct first the following frame using the fixed system of coordinates $(t, x^1, x^2, x^3)$\,.
\begin{definition}\label{definitionofthenullframusingwavecoordinates}
At a point $p$ in the space-time, let
\bea
L &=&  \pa_{t} + \pa_{r} \, =   \pa_{t} +  \frac{x^{i}}{r} \pa_{i} \, , \\
\underline{L} &=&  \pa_{t} - \pa_{r} =  \pa_{t} -  \frac{x^{i}}{r} \pa_{i} \, ,
\eea
and let $\{e_1, e_{2} \}$ be an orthonormal frame on $\SSS^{2}$. 
We define the sets 
\bea
\cal T&=& \{ L,e_1, e_{2}\} \, ,\\
 \cal U&=&\{ \underline{L}, L, e_1, e_{2}\} \, .
\eea
We call the full set ${ \cal U}$ a null-frame. We call the set ${ \cal T}$\, the tangential frame. 
\end{definition}
We also need the following definitions in order to exhibit the non-linear structure of the Einstein-Yang-Mills system in the Lorenz gauge. 
\begin{definition}\label{definitionofthetangentialdefivatives}
We define $\rderm$ as a derivative along any vector $V \in {\cal T }$\,, i.e.
\bea
\rderm := \derm_{V}  \;, \quad \text{with $V \in {\cal{T} }$} \,.
\eea
\end{definition}

\begin{definition}
In our fixed system of coordinates $(t, x^1, x^2, x^3)$\,, we define
\beaa
h_{\mu\nu} := g_{\mu\nu} - m_{\mu\nu} 
\eeaa 
\end{definition}

Now, if we choose that our fixed system of coordinates $(t, x^1, x^2, x^3)$ is \textit{wave coordinates}, then using the null-frame tetrad $\{\, \underline{L}\,, L\,, e_a\,, a \in \{1, 2 \}\, \}$\;, as we showed \cite{G4}, the Einstein-Yang-Mills system in the Lorenz and harmonic gauge is equivalent to the following coupled system of non-linear wave equations on the Einstein-Yang-Mills potential $A$ and on the perturbation metric $g$\,, that that is schematically speaking the following:

\bea\label{waveequationonbadtermfortheEinsteinYangMillspoential}
 \begin{cases}   \notag
   g^{\la\mu} \derm_{\la}   \derm_{\mu}   A_{{\underline{L}}}     =& \derm h  \cdot  \rderm A        + \rderm  h \cdot  \derm A  +    A   \cdot     \rderm A   +  \derm  h  \cdot  A^2   +  A^3 \\
    \notag
  & + O( h \cdot  \derm h \cdot  \derm A) + O( h \cdot  A \cdot \derm A)  + O( h \cdot  \derm h \cdot  A^2) + O( h \cdot  A^3) \\
 & +  \underbrace{A_L   \cdot   \derm A      +  A_{e_a}  \cdot     \derm A_{e_a} }_{\text{bad terms}}    \; , \\
\notag
   g^{\alpha\beta} \derm_\alpha \derm_\beta g_{{\underline{L}} {\underline{L}} }        =&   \rderm h \cdot \derm h +  h \cdot ( \derm h )^2      +   \rderm A    \cdot  \derm  A   +  A^2  \cdot  \derm A     +  A^4 \\
    \notag
     & + O \big(h \cdot  (\derm A)^2 \big)   + O \big(  h  \cdot  A^2 \cdot \derm A \big)     + O \big(  h   \cdot  A^4 \big)   \\
 & + (  \derm h_{\cal TU} ) ^2 +   ( \derm A_{e_{a}} )^2     \; ,\end{cases}
\eea
and for the “good” components of $A$ and $g$\,, we have the “better” structure that is
  \bea
   \begin{cases} 
   \notag
 g^{\la\mu} \derm_{\la}   \derm_{\mu}   A_{{\cal T}}       =&   \derm h  \cdot  \rderm A      +  \rderm  h  \cdot  \derm A   +    A   \cdot   \rderm A  +  \derm  h  \cdot  A^2   +  A^3 \\
 \notag
  & + O( h \cdot  \derm h \cdot  \derm A) + O( h \cdot  A \cdot \derm A)  + O( h \cdot  \derm h \cdot  A^2) + O( h \cdot  A^3)  \; , \\
\notag
    g^{\alpha\beta} \derm_\alpha \derm_\beta g_{ {\cal T} {\cal U}}      =&  \rderm h \cdot \derm h + h \cdot (\derm h)^2      +  \rderm A   \cdot  \derm  A   +    A^2 \cdot \derm A     + A^4  \\
     & + O \big(h \cdot  (\derm A)^2 \big)   + O \big(  h  \cdot  A^2 \cdot \derm A \big)     + O \big(  h   \cdot  A^4 \big)    \; .,\end{cases}
\eea

We see in the troublesome wave equation in \eqref{waveequationonbadtermfortheEinsteinYangMillspoential} on $A_{{\underline{L}}}$\;, that there exist new “bad” terms which are $ A_{e_a}  \cdot     \derm A_{e_a} $\;, where $a \in \{1, 2 \}$\,, and $ A_L   \cdot   \derm A    $\;.

To the best of our knowledge, in the seminal work of Lindblad-Rodnianski \cite{LR10}, their $L^\infty$-estimate does not work to treat these \textit{new terms} of the type $A \cdot \pa A $\,, in contrast to terms of the type $(\pa A)^2$ which were dealt with in \cite{LR10}.

\begin{remark}
The wave equation on the metric that one should actually consider is not that on $\g$\, but rather on $\g - h^0 $\,, where $h^0$ is the spherically symmetric Schwarzschildian part that has infinite energy. However, we still wrote the wave equation on $\g$ because is not our point in this discussion, rather our point is to exhibit the new non-linearity arising from the Yang-Mills potential.
\end{remark}

For clarity, in the coupled system above, the big $O$ notation is defined as in the following Definition.
\begin{definition}\label{definitionofbigOonlyforAandhadgradientofAandgardientofh}
Let $K$ be a tensor that is either $A$ or $h$\;, or $\derm A$\;, or $\derm h$\;. Let $ P_n (K )$ be tensors that are polynomials of degree $n$\;, and $Q_1 (K)$ a tensor that is a Polynomial of degree $1$ such that $Q_1 (0) = 0$ and $Q_1 \neq 0$\;, of which the coefficients are components of the metric $\textbf m$ and of the inverse metric $\textbf m^{-1}$, and of which the variables are components of the covariant tensor $K$, leaving some indices free, so that the following product gives a tensor that we define as,
\bea
O_{\mu_{1} \ldots \mu_{k} } (K  ) &:=& Q_1 ( K) \cdot \Big( \sum_{n=0}^{\infty} P_n ( K ) \Big) \; .
\eea
For a family of tensors $K^{(1)}, \ldots,  K^{(m)}$, where each tensor $K^{(l)}$ is again either $A$\,, or $h$\;, or $\derm A$\;, or $\derm h$\;, we define
\bea
O_{\mu_{1} \ldots \mu_{k} } (K^{(1)} \cdot \ldots \cdot K^{(m)} ) &:=& \prod_{l=1}^{m} Q_{1}^{l} ( K^{(l)} ) \cdot \Big( \sum_{n=0}^{\infty} P_{n}^{l}  ( K^{(l)} ) \Big) \; .
\eea
where again $P_n^{l} (K^{(l)} )$ and $Q_1^l (K^{(l)})$, are tensors that are polynomials of degree $n$ and $1$, respectively, with $Q_1 (0) = 0$ and $Q_1 \neq 0$\;, of which the coefficients are components of the metric $\textbf m$ and of the inverse metric $\textbf m^{-1}$, and of which the variables are components of the covariant tensor $K^{l}$, leaving some indices free, so that at the end the whole product $\prod_{l=1}^{m} Q_{1}^{l} ( K^{(l)} ) \cdot \Big( \sum_{n=0}^{\infty} P_{n}^{l}  ( K^{(l)} ) \Big)$ gives a tensor which we define as $O_{\mu_{1} \ldots \mu_{k} } (K^{(1)} \cdot \ldots \cdot K^{(m)} )$. To lighten the notation, we shall sometimes drop the indices and just write $O (K^{(1)} \cdot \ldots \cdot K^{(m)} )$\;.

\end{definition}

In this paper, we solve this problem by decoupling the higher order energy estimates for each component in a frame, at the level of the $L^2$-norm of the Lie derivatives of each component in the frame ${\cal U}$ of the fields, without involving all the components, so as to address the non-linearity arising from the Yang-Mills fields in the Lorenz gauge. We achieve this \textit{decoupling} of the energy estimates for each component by establishing a new suitable commutator estimate on the higher order energy norm.

The decoupling of the energy estimates in a frame decomposition, that we establish here, namely Theorem \ref{TheSeparateforHigherOrderEnergyEstimate}, are designed to solve problems concerning the new non-linearity arising from the Yang-Mills fields in the Lorenz gauge, such as  $A_{e_a}  \cdot     \derm A_{e_a} $\, for which the $L^\infty$-estimate of Lindblad-Rodnianski does not work, to the best of our knowledge. For this task, we also address the difficulty that the commutator estimate \eqref{commutatorterm} does \textit{not} decouple, a difficulty that we overcome here among other complications that we treat in \cite{G6}. This allows us in \cite{G6} to give the proof of the exterior stability of the $\R^{1+3}$ Minkowski space-time, as solution to the coupled Einstein-Yang-Mills system in the Lorenz gauge, associated to any compact Lie group $G$\,, and in wave coordinates.

For this, we look at the equation 

\bea
g^{\a\b} \derm_{\a } \derm_\b  \Phi_{V} = S^{1}_{V} \; ,
\eea

The metric $\g$ is a perturbation of Minkowski in the following sense: if we define (see Definition \ref{definitionofbigHandsmallhandrecallofdefinitionofMinkowskimetricminrelationtowavecoordiantesasreminder}),
\bea\label{definitionofbigHusingdefinitionofMinkwoskimetric}
H^{\mu\nu} &:=& g^{\mu\nu}-m^{\mu\nu} \;,
\eea
where $m^{\mu\nu}$ is the inverse of the Minkowski metric $m_{\mu\nu}$\;, that is defined to be $(-1, +1, +1, +1)$ in our chosen system of coordinates $(t, x^1, x^2, x^3)$\;, then we assume in our energy estimates that the unknown metric $\g$ is close to Minkowski metric in the following sense:
\bea\label{boundednessoftheperturbationbigHbyaconstant}
 \sum_{\mu, \nu = 0}^{3}   | H_{\mu\nu} | < \frac{1}{n}\; .
\eea
This smallness condition on the perturbation $H$ would make the boundary terms of our energy estimates “look like” the $L^2$-norm of the covariant gradient $ \derm \Phi_{V}  $ of the solution of our system \eqref{nonlinearsystemoftensorialwaveequations} (see Lemma \ref{howtogetthedesirednormintheexpressionofenergyestimate}). This condition could be used in a bootstrap argument to prove that it is actually and ultimately a true estimate.

The goal of our energy estimates is to prove an exterior estimate for each component $ \derm \Phi_{V}  $\;, where $V \in {\cal{U}}$\,, of the higher order energy norm of the tensorial solutions, without involving all the elements of the frame ${\cal U}$\,. We define a product of Minkowski vector fields as in Definitions  \ref{DefinitionofMinkowskivectorfields} and \ref{DefinitionofZI}. By ``decoupled" we mean that would allow us to control the $L^2$ norm of $\derm  \Lie_{Z^I}  \Phi_{V}$, and not only of $\derm  \Lie_{Z^I}  \Phi$\,. By ``exterior" we mean that the integral would be taken on a hypersurface that is the intersection of $t = constant$\;,  in our fixed system of coordinates $(t, x^1, x^2, x^3)$\;, with the complement of the future domain of dependance for the metric $m$ of a compact ${\cal{K}} \subset \Sigma$\;. This is what we mean by decoupled exterior energy estimate: they are bounds on $L^2$ norm of $\derm  \Lie_{Z^I}  \Phi_{V}$ in domains that are exterior to the future domain of dependance of ${\cal K} \subset \Sigma$\;.

We consider a specific non-symmetric tensor (see Definition \ref{defofthestreessenergymomentumtensorforwaveequationhere}) which we contract with a weighted vector (see \eqref{The weightedvectorproportionaltodt}) to get a weighted conservation law (see Lemma \ref{weightedconservationlawintheexteriorwiththeenergymomuntumtensorcontarctedwithvectordt}). Here the weights are defined in Definitions \ref{defoftheweightw}, \ref{defwidehatw} and \ref{defwidetildew}. The fact that the metric $\g$ is assumed to be close enough to the Minkowski metric (see \eqref{boundednessoftheperturbationbigHbyaconstant}), allows us to translate these conservation laws into \textit{weighted} energy estimates as in Corollary \ref{TheenerhyestimatewithtermsinvolvingHandderivativeofHandwithwandhatw}. Thanks to the definition of our tensor $T$ in Definition \ref{defofthestreessenergymomentumtensorforwaveequationhere}, we get a suitable control on the $L^2$-norm on each component, namely,
  \bea
  \notag
     \int_{\Sigma^{ext}_{t} }  | \Lie_{Z^I}  \Phi_{V}|^2     \cdot w(q)  \cdot d^{n}x   \; ,
  \eea
  
 a control that will be used in a subsequent paper, \cite{G6}, to prove the non-linear stability of the $(1+3)$-Minkowski space-time governed by the coupled Einstein-Yang-Mills system in the Lorenz gauge. Here $q$ is defined as in Definition \ref{definitionoftheparmeterq}.

The fact that we can decouple the controls for the $L^2$-norm on each component $\derm \Lie_{Z^I}  \Phi_{V} $ is necessary to treat the term $ A_{e_a}  \cdot     \derm A_{e_a} $\,. However, to successfully achieve this, we crucially need to deal with the fact that the commutator estimate does not decouple, i.e. with the fact that the estimate on the term
\bea\label{commutatorterm}
 |g^{\a\b} \derm_{\a } \derm_\b \Lie_{Z^I}  \Phi_{V} - \Lie_{Z^I} ( g^{\a\b} \derm_{\a } \derm_\b  \Phi_{V} ) | \; ,
 \eea
contains terms that involve other components different than $\Lie_{Z^I}  \Phi_{V} $. We provide a suitable new commutator estimate in \eqref{Theseperatecommutatortermestimateforthetangentialcomponents}. Indeed, unlike the case of the Einstein vacuum equations, and also unlike the case of the Einstein-Maxwell system, in the case of the Einstein-Yang-Mills equations, we need to have a \textit{decoupled} estimate for \textit{each component} of the following commutator term $$| \Lie_{Z^I}  ( g^{\la\mu} \derm_{\la}   \derm_{\mu}     A_{e_a} ) - g^{\la\mu}    \derm_{\la}   \derm_{\mu}  (  \Lie_{Z^I} A_{e_a}  ) | \; .$$

We solve the problem by establishing the following suitable estimate.

 \bea\label{Theseperatecommutatortermestimateforthetangentialcomponents}
\notag
&&| g^{\la\mu}    \derm_{\la}   \derm_{\mu} \Lie_{Z^I} \Phi_{V}    - \Lie_{Z^I}  ( g^{\la\mu} \derm_{\la}   \derm_{\mu}     \Phi_{V} )  |  \\
  \notag
   &\les&  \sum_{|K| < |I| }  | g^{\la\mu} \cdot \derm_{\la}   \derm_{\mu} (  \Lie_{Z^{K}}  \Phi_{V} ) | \\
   \notag
&&+  \frac{1}{(1+t+|q|)}  \cdot \sum_{|K|\leq |I|,}\,\, \sum_{|J|+(|K|-1)_+\le |I|} \,\,\, | \Lie_{Z^{J}} H |\, \cdot | \derm ( \Lie_{Z^K}  \Phi )  | \\
&& + \underbrace{  \frac{1}{(1+|q|)}}_{\text{bad factor}}  \cdot \sum_{|K|\leq |I|,}\,\, \sum_{|J|+(|K|-1)_+\le |I|} \,\,\, | \Lie_{Z^{J}} H_{L  L} |\, \cdot  \underbrace{\big( \sum_{ V^\prime \in \cal T } | \derm ( \Lie_{Z^K}  \Phi _{V^\prime} )  | \:  \big) }_{\text{decoupled tangential components}} \; ,
\eea
  where $(|K|-1)_+=|K|-1$ if $|K|\geq 1$ and $(|K|-1)_+=0$ if $|K|=0$\,. 
  
  With this \textit{new} estimate \eqref{Theseperatecommutatortermestimateforthetangentialcomponents}, we notice that the terms with the bad factor $\frac{1}{(1+|q|)}$ enter \textit{only} with the tangential components. This is in fact the bad term in the estimate of the commutator term, that is behind an imposed troublesome Grönwall inequality in the proof in \cite{G6}, that needs to be dealt with separately. The estimate used previously in the literature contains all the components for the factor  $\frac{1}{(1+|q|)}$\,, and does \textit{not} allow us to decouple the higher order energy estimates.
  
Instead, we are able to establish this new estimate \eqref{Theseperatecommutatortermestimateforthetangentialcomponents} thanks to the following inequality (see \eqref{estimateforgradientoftensorestaimetedbyLiederivativesoftensorwithrefinedcomponentsforthebadfactor}), 
  
 \bea
 |\derm \Psi_{UV} | \les \sum_{|I| \leq 1}  \frac{1}{(1+t+|q|)} \cdot | \Lie_{Z^I} \Psi |  +  \sum_{U^\prime \in \cal U,  V^\prime \in \cal T } \sum_{|I| \leq 1}  \frac{1}{(1+|q|)} \cdot  | \Lie_{Z^I} \Psi_{U^\prime V^\prime} | \; . 
 \eea

  Whereas the terms that contain the full components in \eqref{Theseperatecommutatortermestimateforthetangentialcomponents}, have the good factor $\frac{1}{(1+t+|q|)} $.
  Hence, we get an estimate on the commutator term of the higher order energy, where term with the weak factor is insensitive to the bad component $A_{{\underline{L}}}$ that bothered us (as explained above).

Furthermore, we get a weighted conservation law (see Corollary \ref{WeightedenergyestimateusingHandnosmallnessyet}) that has also a quantity that controls the tangential derivatives of our unknown solution $ \Lie_{Z^I}  \Phi_{V}$ (see Lemma \ref{expressionofTttplusTtr}), thanks to the non-vanishing weight $\widehat{w}^\prime$ (see Lemma \ref{equaivalenceoftildewandtildeandofderivativeoftildwandderivativeofhatw}). In other words, thanks to the non-trivial weight $\widehat{w}$ (see Definition \ref{defwidehatw}), we have control on the weighted space-time integral of the covariant tangential derivatives of $\Phi_{V}$\;, namely
\bea
\notag
 \int_{t_1}^{t_2}  \int_{\Sigma^{ext}_{\tau} }  \Big(    \frac{1}{2} \Big(  | (\derm_t   + \derm_r  ) \Lie_{Z^I}  \Phi_{V} |^2  +  \de^{ij}  | ( \derm_i - \frac{x_i}{r} \derm_{r}  ) \Lie_{Z^I}  \Phi_{V} |^2 \Big)  \cdot d\tau \cdot \widehat{w}^\prime (q) d^{n}x \;.
\eea

However, for the energy estimate to be successful, it is of crucial importance that the bounds in the space-time integrand on the right hand side of the inequality are the correct bounds. It is the case in our bounds since we have the following 

\beaa
    \notag
     && \int_{t_1}^{t_2}  \int_{\Sigma^{ext}_{\tau} }   \overbrace{| H_{LL } |}^{\text{good component}} \cdot  \overbrace{| \derm  \Lie_{Z^I}  \Phi_{V} |^2 }^{\text{decoupled bad derivatives}}     \cdot d\tau \cdot   \widetilde{w} ^\prime (q) d^{3}x \\
  \notag   
     && + \int_{t_1}^{t_2}  \int_{\Sigma^{ext}_{\tau} }  \overbrace{|  H |}^{\text{bad component}}  \cdot    \overbrace{ | \rderm   \Lie_{Z^I}  \Phi_{V} | }^{\text{good derivative}}  \cdot  \overbrace{| \derm  \Lie_{Z^I}  \Phi_{V} |}^{\text{decoupled bad derivative}}  \Big)  \cdot d\tau\cdot   \widetilde{w} ^\prime (q) d^{3}x \; .  \\
 \notag
&& + \int_{t_1}^{t_2}  \int_{\Sigma^{ext}_{\tau} } \Big( \overbrace{ | \derm H_{LL} | }^{\text{good component}} +  \overbrace{| \rderm H  |}^{\text{good derivative}}  \Big) \cdot  \overbrace{| \derm \Phi_{V} |^2}^{\text{decoupled bad derivatives}} \cdot d\tau \cdot   \widetilde{w} (q) d^{3}x \\
     \notag
     && + \int_{t_1}^{t_2}  \int_{\Sigma^{ext}_{\tau} }  \underbrace{ | \derm H |}_{\text{bad derivative}}   \cdot    \underbrace{| \rderm  \Lie_{Z^I}  \Phi_{V}  |}_{\text{good derivative}}  \cdot   \underbrace{ | \derm  \Lie_{Z^I}  \Phi_{V} |}_{\text{decoupled bad derivative}} \Big)  \cdot d\tau \cdot   \widetilde{w} (q) d^{3}x \; .\\
\eeaa
Also, in the commutator estimate \eqref{Theseperatecommutatortermestimateforthetangentialcomponents}, we have

 \bea
\notag
&&  \overbrace{\frac{1}{(1+t+|q|)} }^{\text{good factor}}  \cdot \sum_{|K|\leq |I|,}\,\, \sum_{|J|+(|K|-1)_+\le |I|} \,\,\,  \overbrace{ | \Lie_{Z^{J}} H |\, \cdot | \derm ( \Lie_{Z^K}  \Phi )  |}^{\text{bad components}} \\
&& + \underbrace{  \frac{1}{(1+|q|)}}_{\text{bad factor}}  \cdot \sum_{|K|\leq |I|,}\,\, \sum_{|J|+(|K|-1)_+\le |I|} \,\,\,  \underbrace{  | \Lie_{Z^{J}} H_{L  L} |}_{\text{good component}} \, \cdot  \underbrace{\big( \sum_{ V^\prime \in \cal T } | \derm ( \Lie_{Z^K}  \Phi _{V^\prime} )  | \:  \big) }_{\text{decoupled tangential components}} \; .
\eea

  Finally, we would have established the desired decoupled higher order energy estimates.

\subsection{The statement}\

In order to state the theorem, we shall make the following definitions and notations. 
\begin{definition}\label{definitionofMinkowskiandcovariantderivtaiveofMinkwoforafixedgivensystemofcoordinates}
Let  $(x^0, x^1, x^2, x^3)$ be a fixed system of coordinates, where we shall also write, sometimes, $x^0 = t$\;. We define $m$ be the Minkowski metric $(-1, +1, +1, +1)$ in this fixed system of coordinates. We define $\derm$ be the covariant derivative associate to the metric $m$\;..

\end{definition}

\begin{definition}
For an arbitrary tensor or arbitrary order, say $K_{\a\b}$, we define
 \bea
 \notag
|  K |^2 &:=&  \sum_{\a,\; \b \in   \{t, x^1, x^2, x^3 \}}  |  K_{\a\b} |^2   \, .
\eea
Since $\derm  K_{UV}$ is tensor, and $\derm  K$ is a tensor of higher order, the definition of the norm gives
 \bea
 \notag
|  \derm  K_{UV} |^2 &=&  \sum_{ \mu \in  \{t, x^1, x^2, x^3 \}} |  \derm_{\mu} K_{UV} |^2   \, , \\
 \notag
|  \derm  K |^2 &=&  \sum_{\a,\; \b,\, \mu \in  \{t, x^1, x^2, x^3 \}} |  \derm_{\mu} K_{\a\b} |^2   \, .
\eea
We define the norm of the tangential covariant derivative as 
 \bea
 \notag
|  \rderm  K |^2 &=&  \sum_{U \in  {\cal T}}   \sum_{\;  \a,\, \b \in  \{t, x^1, x^2, x^3 \}} |  \derm_{U} K_{\a\b} |^2   \, ,
\eea
where ${\cal T}$ is defined in Definition \ref{definitionofthenullframusingwavecoordinates}.
\end{definition}

\begin{definition} \label{DefinitionofMinkowskivectorfields}
Let $ x_{\b} = m_{\mu\b} x^{\mu} \;, $ \,$Z_{\a\b} = x_{\b} \pa_{\a} - x_{\a} \pa_{\b}  \;, $ and $S = t \pa_t + \sum_{i=1}^{n} x^i \pa_{i}  \; .$ The Minkowski vector fields are the vectors of the following set
\bea
{\cal Z}  := \big\{ Z_{\a\b}\;,\; S\;,\; \pa_{\a} \, \,  | \, \,   \a\;,\; \b \in \{ 0, 1, 2, 3 \} \big\}  \; .
\eea
Vectors belonging to ${\cal Z}$ will be denoted by $Z$\;.
\end{definition}

\begin{definition} \label{DefinitionofZI}
We define
\bea
Z^I :=Z^{\iota_1} \ldots Z^{\iota_k} \quad \text{for} \quad I=(\iota_1, \ldots,\iota_k),  
\eea
where $\iota_i$ is an $11$-dimensional integer index, with $|\iota_i|=1$, and $Z^{\iota_i}$ representing each a vector field from the family ${\cal Z}$. For a tensor $T$, of arbitrary order, either a scalar or valued in the Lie algebra, we define the Lie derivative as
\bea
\Lie_{Z^I} T :=\Lie_{Z^{\iota_1}} \ldots \Lie_{Z^{\iota_k}} T \quad \text{for} \quad I=(\iota_1, \ldots,\iota_k) .
\eea
\end{definition}

We are now ready to state our weighted decoupled energy estimate in the following theorem.

 \begin{theorem}\label{TheSeparateforHigherOrderEnergyEstimate}
Let the weights be as in Definitions \ref{defoftheweightw}, \ref{defwidehatw} and \ref{defwidetildew}. For solutions of the dynamically coupled wave equation \eqref{nonlinearsystemoftensorialwaveequations} ,with $H$ defined in \eqref{definitionofbigHusingdefinitionofMinkwoskimetric}, satisfying
\bea
| H| < \frac{1}{3}\; ,
\eea
and for $\Phi$ decaying sufficiently fast at spatial infinity, we have for all $V \in {\cal U}$\,, that is in the frame decomposition of the tensorial equations, the following estimate
  \bea\label{TheenerhyestimatewithtermsinvolvingHandderivativeofHandwithwandhatw}
  \notag
 &&     \int_{\Sigma^{ext}_{t_2} }  |\derm \Lie_{Z^I}  \Phi_{V} |^2     \cdot w(q)  \cdot d^{n}x   + \int_{t_1}^{t_2}  \int_{\Sigma^{ext}_{\tau} }   | \rderm \Lie_{Z^I}  \Phi_{V} |^2   \cdot d\tau \cdot \widehat{w}^\prime (q) d^{n}x \\
  \notag 
   &\les &       \int_{\Sigma^{ext}_{t_1} }  |\derm  \Lie_{Z^I}  \Phi_{V} |^2     \cdot   \widetilde{w} (q)  \cdot d^{3}x \\
    \notag
     && +  \int_{t_1}^{t_2}  \int_{\Sigma^{ext}_{\tau} }  | H_{LL } | \cdot | \derm  \Lie_{Z^I}  \Phi_{V} |^2    \cdot d\tau \cdot   \widetilde{w} ^\prime (q) d^{3}x \\
  \notag   
     && + \int_{t_1}^{t_2}  \int_{\Sigma^{ext}_{\tau} }  |  H |  \cdot   | \rderm   \Lie_{Z^I}  \Phi_{V} |  \cdot  | \derm  \Lie_{Z^I}  \Phi_{V} | \Big)  \cdot d\tau\cdot   \widetilde{w} ^\prime (q) d^{3}x \; .  \\
 \notag
&& + \int_{t_1}^{t_2}  \int_{\Sigma^{ext}_{\tau} } \Big(  | \derm H_{LL} |  +  | \rderm H  |  \Big) \cdot | \derm \Phi_{V} |^2 \cdot d\tau \cdot   \widetilde{w} (q) d^{3}x \\
     \notag
     && + \int_{t_1}^{t_2}  \int_{\Sigma^{ext}_{\tau} }  | \derm H |  \cdot    | \rderm  \Lie_{Z^I}  \Phi_{V}  |  \cdot  | \derm  \Lie_{Z^I}  \Phi_{V} | \Big)  \cdot d\tau \cdot   \widetilde{w} (q) d^{3}x \; \\
   &&  +  \int_{t_1}^{t_2}  \int_{\Sigma^{ext}_{\tau} }  |  g^{\mu\a} \derm_{\mu } \derm_\a  \Lie_{Z^I}  \Phi_{V} | \cdot |  \derm_t  \Lie_{Z^I}  \Phi_{V} |  \cdot d\tau \cdot   \widetilde{w} (q) d^{3}x  \; .
 \eea

 We also have the following decoupled estimate on the tangential components of the commutator term, which could be good components that satisfy a better wave equations with a good non-linearity. For any $V \in \cal T$,
 \bea\label{Theseperatecommutatortermestimateforthetangentialcomponents}
\notag
&&| g^{\la\mu}    \derm_{\la}   \derm_{\mu} \Lie_{Z^I} \Phi_{V}    - \Lie_{Z^I}  ( g^{\la\mu} \derm_{\la}   \derm_{\mu}     \Phi_{V} )  |  \\
  \notag
   &\les&  \sum_{|K| < |I| }  | g^{\la\mu} \cdot \derm_{\la}   \derm_{\mu} (  \Lie_{Z^{K}}  \Phi_{V} ) | \\
   \notag
&&+  \frac{1}{(1+t+|q|)}  \cdot \sum_{|K|\leq |I|,}\,\, \sum_{|J|+(|K|-1)_+\le |I|} \,\,\, | \Lie_{Z^{J}} H |\, \cdot | \derm ( \Lie_{Z^K}  \Phi )  | \\
&& +   \frac{1}{(1+|q|)}  \cdot \sum_{|K|\leq |I|,}\,\, \sum_{|J|+(|K|-1)_+\le |I|} \,\,\, | \Lie_{Z^{J}} H_{L  L} |\, \cdot  \big( \sum_{ V^\prime \in \cal T } | \derm ( \Lie_{Z^K}  \Phi _{V^\prime} )  | \:  \big) \; ,
\eea
  where $(|K|-1)_+=|K|-1$ if $|K|\geq 1$ and $(|K|-1)_+=0$ if $|K|=0$\,. For any $V \in \cal U$\,, the estimate in \eqref{Theseperatecommutatortermestimateforthetangentialcomponents} holds with the sum $\sum_{ V^\prime \in \cal T }$ replaced by $ \sum_{ V^\prime \in \cal U }$\,.

\begin{remark}
Here $\Sigma^{ext}$ is defined in Definition \ref{definitionoftheexteriorslicesigmat}. We also have the non-exterior version of this Theorem \ref{TheenerhyestimatewithtermsinvolvingHandderivativeofHandwithwandhatw}, where the integration is taken on the whole slice $\Sigma_{t} $ prescribed by constant $t $ hypersurfaces. 
\end{remark}
\begin{remark}
The same estimates hold for any tensor of arbitrary order. In that case, for each vector $V \in {\cal T}$ appearing on the left hand side of \eqref{Theseperatecommutatortermestimateforthetangentialcomponents}, it has to be replaced on the right hand side of \eqref{Theseperatecommutatortermestimateforthetangentialcomponents} by the $\sum_{ V^\prime \in \cal T }$\,, and for each vector  $V \in {\cal U}$, it has to be replaced on the right hand side of \eqref{Theseperatecommutatortermestimateforthetangentialcomponents} by the $\sum_{ V^\prime \in {\cal U} }$\,.
\end{remark}

\end{theorem}

 \section{Energy estimates in the exterior region}
 
 The goal of this section is to prove \eqref{TheenerhyestimatewithtermsinvolvingHandderivativeofHandwithwandhatw} in Theorem \ref{TheenerhyestimatewithtermsinvolvingHandderivativeofHandwithwandhatw}.

\subsection{The definitions and notations}\

 \begin{definition}\label{defofthestreessenergymomentumtensorforwaveequationhere}
 Let either $\Phi$ be a tensor of arbitrary order. For simplicity of notation, we consider $\Phi$ to be a one-tensor. We consider the following non-symmetric tensor for wave equations. We define
\bea\label{definitionofthewavestreessenergymomentumtensorfortensorphiVwhereVisanyvector}
\notag
T^{(\bf{g}) \; \mu}_{\;\;\;\;\;\;\;\;\;\;\nu} ( \Lie_{Z^I}  \Phi_{V} )    =  g^{\mu\a}< \derm_\a  \Lie_{Z^I}  \Phi_{V} ,   \derm_\nu  \Lie_{Z^I}  \Phi_{V}> - \frac{1}{2} m^{\mu}_{\;\;\;\nu}  \cdot g^{\a\b}  < \derm_\a  \Lie_{Z^I}  \Phi_{V},   \derm_\b  \Lie_{Z^I}  \Phi_{V}> \;, \\
 \eea
 where we raise index with respect to the Minkowski metric $m$\,, defined to be Minkowski in wave coordinates. We consider $\Phi$ to be a field decaying fast enough at spatial infinity.
 \end{definition}

\begin{definition}\label{definitionofLtilde}
We define $ \widetilde{L}^{\nu} $ as a vector at a point in the space-time, such that  $T_{\widetilde{L} t}^{(\bf{g})} \geq 0$\;. We note that such a definition does not give a unique vector, however, we will use this to construct $N_{t_1}^{t_2} (q_0) $ in Definition \ref{definitionofNandofNtruncatedbtweentwots} and to define $ \hat{L}^{\nu} $ in Definition \ref{definitionofwidetildeLfordivergencetheoremuse}.
 \end{definition}

\begin{definition}\label{definitionofNandofNtruncatedbtweentwots}
We define $N_{t_1}^{t_2} (q_0) $ as a boundary made by the following:

For a $t_N \geq 0 $\;, we take a curve in the plane $(t, r)$\;, at an $\Om \in \SSS^{n-1}$ fixed, that starts at $(t=t_N\,,\, r=0) $ and extends to the future of $\Sigma_{t_0}$\;, and where $ \widetilde{L}^{\nu} $ (defined in Definition \ref{definitionofLtilde}) is orthogonal (with respect to the Minkowski metric $m$) to that curve at each point. For a $t_N \geq 0 $ large enough, depending on $q_0$\;, there exits such a curve that is contained in the region $\{(t, x) \;|\; q:= r-t \leq q_0  \;\,,\; t \leq t_2 \}$\;.

We define $N (q_0)$ as the product of $\SSS^{n-1}$ and of such future inextendable curve, and we define $N_{t_1}^{t_2} (q_0) $ as $N (q_0) $ truncated between $t_1$ and $t_2$\;. Thus,  $N (q_0)$ depends on the choice of the starting point $(t_N, 0)$ that is the tip of $N (q_0)$\;, however we write $N(q_0)$ to refer to a choice of such $N(q_0)$\;. To lighten the notation, we write $N$\;, instead of $N(q_0)$\;, where a $q_0$ has been chosen fixed. 

\end{definition} 

\begin{definition} Let $\overline{C}$ be the exterior region defined as the following:

The boundary $N$ (given in Definition \eqref{definitionofNandofNtruncatedbtweentwots}) separates the space-time into two regions: one interior region that we shall call $C$ where space-like curves are contained in a compact region, and the other region is the complement of $C$\;, where space-like curves can go to spatial infinity. We define $\overline{C}$ as the complement of $C$\;. 

\end{definition} 

\begin{definition}\label{definitionoftheexteriorslicesigmat}
 We define for a given fixed $t$\;,
   \bea
\Sigma^{ext}_{t}  &:=&  \Sigma_t  \cap \overline{C} \, .
\eea
\end{definition}

\begin{definition} 

We define $n^{(\bf{m}), \mu}_{\Sigma}$ as the future oriented unit orthogonal vector (for the metric $m$) to the hypersurface $\Sigma^{ext}_{t}$\;, and $dv^{(\bf{m})}_\Sigma$ as the induced volume form on $\Sigma_t$\;.

We denote by $n^{(\bf{m}), \mu}_{N}$ an orthogonal vector (for the metric $m$) to the hypersurface $N_{t_1}^{t_2}$\;,  and by $dv^{(\bf{m})}_N$ a volume form on $N_{t_1}^{t_2} $ such that the divergence theorem applies.

\end{definition} 

\begin{definition}\label{definitionofwidetildeLfordivergencetheoremuse}
We define $ \hat{L}^{\nu} $ as a vector proportional to $ \widetilde{L}^{\nu} $ (defined in Definition \ref{definitionofLtilde}), and therefore orthogonal (with respect to the Minkowski metric $m$) to $N$ (that is defined in Definition \ref{definitionofNandofNtruncatedbtweentwots}), and such that $\hat{L}$ is oriented,  and with Euclidian length, in such a way that the stated divergence theorem in what follows hold true: see, for instance, Lemmas \ref{conservationlawwithoutcomputingdivergenceofTmut} and \ref{weightedconservationlawintheexteriorwiththeenergymomuntumtensorcontarctedwithvectordt} and Corollary \ref{WeightedenergyestimateusingHandnosmallnessyet}.
 \end{definition}

\begin{remark}
Based on the construction of $N(q_0)$ in Definition \ref{definitionofNandofNtruncatedbtweentwots}, we have that the exterior region $\overline{C} $ includes the region $\{(t, x) \;|\; q:= r- t \geq q_0 \}$\;.
\end{remark}

    \begin{definition}\label{definitionofbigHandsmallhandrecallofdefinitionofMinkowskimetricminrelationtowavecoordiantesasreminder}
We define $H$ as the 2-tensor given by
\bea
H^{\mu\nu} &:=& g^{\mu\nu}-m^{\mu\nu} \;,
\eea
where $m^{\mu\nu}$ is the inverse of the Minkowski metric $m_{\mu\nu}$\;, defined in Definition \ref{definitionofMinkowskiandcovariantderivtaiveofMinkwoforafixedgivensystemofcoordinates}. 
In addition, we define
\bea
h_{\mu\nu} &:=& g_{\mu\nu} - m_{\mu\nu} \; , \\
h^{\mu\nu} &:=& m^{\mu\mu^\prime}m^{\nu\nu^\prime}h_{\mu^\prime\nu^\prime} \;.
\eea
\end{definition}

\subsection{Conservation laws for wave equations}\
 
\begin{lemma}\label{generalconservationlawwithanyvectorfieldX}
For a vector field $X^{\nu}$\,, we have
 \bea
 \notag
&& \int_{t_1}^{t_2}  \int_{\Sigma^{ext}_{\tau}} \derm^{\mu} \big( X^{\nu}T_{\mu\nu}^{(\bf{g})} \big)  \cdot dv^{(\bf{m})} \\
 \notag
 &=&  \int_{t_1}^{t_2}  \int_{\Sigma^{ext}_{\tau} } \Big( \big( \derm^{\mu} X^{\nu}  \big) \cdot T_{\mu\nu}^{(\bf{g})}  + X^{\nu}< g^{\mu\a} \derm_{\mu } \derm_\a \Lie_{Z^I}  \Phi_{V} ,   \derm_\nu \Lie_{Z^I}  \Phi_{V}  >\\
  \notag
&& +( \derm_{\mu } H^{\mu\a} ) \cdot X^{\nu} < \derm_\a \Lie_{Z^I}  \Phi_{V}  ,   \derm_\nu \Lie_{Z^I}  \Phi_{V}  > \\
\notag
&& - \frac{1}{2} X^{\nu} m^{\mu}_{\;\;\;\nu}  \cdot  ( \derm_{\mu } H^{\a\b} ) \cdot < \derm_\a \Lie_{Z^I}  \Phi_{V}  ,   \derm_\b \Lie_{Z^I}  \Phi_{V}  >   \Big) \cdot dv^{(\bf{m})}  \,,
\eea
 where the tensor $T_{\mu\nu}^{(\bf{g})} $ is defined in Definition \ref{defofthestreessenergymomentumtensorforwaveequationhere}.
\end{lemma}

\begin{proof}
Contracting the stress-energy-momentum tensor with a vector field $X^{\nu}$\,, and applying the divergence theorem to $X^{\nu}T_{\mu\nu}$\,, one gets
 \beaa
 \notag
&&  \int_{\Sigma^{ext}_{t_2} }  \big( X^{\nu}T_{\mu\nu}^{(\bf{g})} \big)  n^{(\bf{g}), \mu}_{\Sigma} \cdot dv^{(\bf{g})}_\Sigma +   \int_{N_{t_1}^{t_2} }  \big( X^{\nu}T_{\mu\nu}^{(\bf{g})} \big)  n^{(\bf{g}), \mu}_{N} \cdot dv^{(\bf{g})}_N + \int_{t_1}^{t_2}  \int_{\Sigma^{ext}_{\tau} (q)} \derm^{\mu} \big( X^{\nu}T_{\mu\nu}^{(\bf{g})} \big)  \cdot dv^{(\bf{g})} \\
 \notag
 &=&   \int_{\Sigma^{ext}_{t_1} }  \big( X^{\nu}T_{\mu\nu}^{(\bf{g})} \big)  n^{(\bf{g}), \mu}_{\Sigma} \cdot dv^{(\bf{g})}_\Sigma  \; . 
 \eeaa
We then compute 
\beaa
&& \derm^{\mu }T^{(\bf{g})}_{\mu\nu} \\
&:=& m^{\mu\la}\derm_{\la }T_{\mu\nu} =\derm_{\mu } T^{\mu}_{\;\;\;\nu} \;  \\
&=& ( \derm_{\mu } g^{\mu\a} ) \cdot < \derm_\a \Lie_{Z^I}  \Phi_{V} ,   \derm_\nu \Lie_{Z^I}  \Phi_{V}> \\
&& - \frac{1}{2} m^{\mu}_{\;\;\;\nu}  \cdot  ( \derm_{\mu } g^{\a\b} ) \cdot < \derm_\a \Lie_{Z^I}  \Phi_{V} ,   \derm_\b \Lie_{Z^I}  \Phi_{V}> \\
&& + < g^{\mu\a} \derm_{\mu } \derm_\a \Lie_{Z^I}  \Phi_{V} ,   \derm_\nu \Lie_{Z^I}  \Phi_{V} > \\
&& +  g^{\mu\a}< \derm_\a \Lie_{Z^I}  \Phi_{V} ,  \derm_{\mu } \derm_\nu \Lie_{Z^I}  \Phi_{V} >   -   g^{\a\b}  <\derm_{\nu } \derm_\a \Lie_{Z^I}  \Phi_{V} ,   \derm_\b \Lie_{Z^I}  \Phi_{V}> \; .
\eeaa
Now, we can compute in wave coordinates $\derm_{\nu } \derm_\a \Phi  $, and if the end result for $\derm^{\mu }T_{\mu\nu}$ gives a tensor in $\nu$, then the identity that we obtain will be true independently of the system of coordinates. In wave coordinates, the Christoffel symbols are vanishing and therefore, the two derivates commute, i.e. $ \derm_{\nu } \derm_\a \Phi  = \derm_{\a } \derm_\nu \Phi \; .$ Thus, in wave coordinates
\beaa
&& \derm^{\mu }T^{(\bf{g})}_{\mu\nu}  \\
&=&( \derm_{\mu } g^{\mu\a} ) \cdot < \derm_\a \Lie_{Z^I}  \Phi_{V} ,   \derm_\nu\Lie_{Z^I}  \Phi_{V} >  \\
&& - \frac{1}{2} m^{\mu}_{\;\;\;\nu}  \cdot  ( \derm_{\mu } g^{\a\b} ) \cdot < \derm_\a \Lie_{Z^I}  \Phi_{V},   \derm_\b \Lie_{Z^I}  \Phi_{V}>  + < g^{\mu\a} \derm_{\mu } \derm_\a \Lie_{Z^I}  \Phi_{V} ,   \derm_\nu \Lie_{Z^I}  \Phi_{V} > \; .
\eeaa
 Since the end result is a tensor in $\nu$\:, we obtain
 \beaa
\derm^{\mu }T^{(\bf{g})}_{\mu\nu} &=&( \derm_{\mu } g^{\mu\a} ) \cdot < \derm_\a \Lie_{Z^I}  \Phi_{V} ,   \derm_\nu \Lie_{Z^I}  \Phi_{V} > \\
&&  - \frac{1}{2} m^{\mu}_{\;\;\;\nu}  \cdot  ( \derm_{\mu } g^{\a\b} ) \cdot < \derm_\a \Lie_{Z^I}  \Phi_{V},   \derm_\b \Lie_{Z^I}  \Phi_{V}> \\
&& + < g^{\mu\a} \derm_{\mu } \derm_\a \Lie_{Z^I}  \Phi_{V} ,   \derm_\nu \Lie_{Z^I}  \Phi_{V} > \; .
\eeaa
Contracting the stress-energy-momentum tensor with respect to the second index, with a vector field $X^{\nu}$, and computing the covariant divergence of $X^{\nu}T_{\mu\nu}$, one gets
\bea
\derm^{\mu} \big( X^{\nu}T_{\mu\nu}^{(\bf{g})} \big) &=& \derm^{\mu} X^{\nu}  \cdot T_{\mu\nu}^{(\bf{g})} +  \derm^{\mu }T^{(\bf{g})}_{\mu X}  \; .
\eea
We can then write 
 \beaa
 \notag
&&\derm^{\mu} \big( X^{\nu}T_{\mu\nu}^{(\bf{g})} \big) \\
\notag
&=& \big( \derm^{\mu} X^{\nu}  \big) \cdot T_{\mu\nu}^{(\bf{g})} +X^{\nu}   < g^{\mu\a} \derm_{\mu } \derm_\a \Lie_{Z^I}  \Phi_{V},   \derm_\nu \Lie_{Z^I}  \Phi_{V}>\\
 \notag
&& +( \derm_{\mu } g^{\mu\a} ) \cdot X^{\nu}  < \derm_\a \Lie_{Z^I}  \Phi_{V} ,   \derm_\nu \Lie_{Z^I}  \Phi_{V} > \\
&& - \frac{1}{2} X^{\nu}  m^{\mu}_{\;\;\;\nu}  \cdot  ( \derm_{\mu } g^{\a\b} ) \cdot < \derm_\a \Lie_{Z^I}  \Phi_{V},   \derm_\b\Lie_{Z^I}  \Phi_{V}> \; . 
\eeaa
Using the definition of $H^{\mu\nu}:=g^{\mu\nu}-m^{\mu\nu}$, we get the result
\end{proof}
We recall that $ \hat{L}^{\nu} $ is defined as in Definition \ref{definitionofwidetildeLfordivergencetheoremuse}.
\begin{lemma}\label{conservationlawwithoutcomputingdivergenceofTmut}
We have
  \beaa
  \notag
   &&   \int_{\Sigma^{ext}_{t_2} } \Big(  - \frac{1}{2} g^{t t}< \derm_t \Lie_{Z^I}  \Phi_{V} ,   \derm_t \Lie_{Z^I}  \Phi_{V} >   + \frac{1}{2} g^{j i}  < \derm_j \Lie_{Z^I}  \Phi_{V} ,   \derm_i \Lie_{Z^I}  \Phi_{V} > \Big)    \cdot d^{n}x \\
   \notag
&& +  \int_{N_{t_1}^{t_2} }  \big( T_{\hat{L} t}^{(\bf{g})} \big)  \cdot dv^{(\bf{m})}_N \\
     \notag
   &=&   \int_{\Sigma^{ext}_{t_1} }   \Big(  - \frac{1}{2} g^{t t}< \derm_t \Lie_{Z^I}  \Phi_{V} ,   \derm_t \Lie_{Z^I}  \Phi_{V} >   + \frac{1}{2} g^{j i}  < \derm_j \Lie_{Z^I}  \Phi_{V},   \derm_i \Lie_{Z^I}  \Phi_{V} >  \Big)  \cdot d^{n}x  \\
     \notag
     &&- \int_{t_1}^{t_2}  \int_{\Sigma^{ext}_{\tau} } \Big( \derm^{\mu }T_{\mu t}^{(\bf{g})}     \Big) \cdot d\tau d^{n}x   \; .
 \eeaa
\end{lemma}

\begin{proof}
 Considering the metric $m$, we know by definition of $m$ being the Minkowski metric in wave coordinate $\{t, x^1, \ldots, x^n \}$, that for $X = \frac{\pa}{\pa t} $\:, we then have
 \beaa
 \big( \derm^{\mu}  \big( \frac{\pa}{\pa t}  \big)^{\nu} \big) \cdot T_{\mu\nu}^{(\bf{g})}   &=&  0 \; , \\
n^{(\bf{m}), \mu}_{\Sigma}  =  \big( \frac{\pa}{\pa t}  \big)^{\mu} \; & ,& \quad dv^{(\bf{m})}_\Sigma = dx^1 \ldots dx^n  := d^{n} x \; .
 \eeaa
Hence, the conservation law in Lemma \ref{generalconservationlawwithanyvectorfieldX}, obtained through the divergence theorem for the non-symmetric tensor $T_{\mu t}$, gives
 \bea\label{conservationlawforvectorfieldt}
 \notag
  \int_{t_1}^{t_2}  \int_{\Sigma^{ext}_{\tau} } \Big(   \derm^{\mu }T_{\mu t}^{(\bf{g})}    \Big) \cdot d\tau d^{n}x  &=&   \int_{\Sigma^{ext}_{t_2} }  T_{ t t}^{(\bf{g})}   \cdot d^{n}x -  \int_{\Sigma^{ext}_{t_1} }   T_{t t}^{(\bf{g})} \cdot d^{n}x   -  \int_{N_{t_1}^{t_2} }  \big( T_{\hat{L} t}^{(\bf{g})} \big)  \cdot dv^{(\bf{m})}_N \; .\\
 \eea
We compute
 \bea\label{evlautingTttforourenergymomentumtensor}
\notag
- T_{ t t}^{(\bf{g})}  &=&  {T^{(\bf{g})}}^{t}_{\;\;\; t} \; \\
\notag
&=&  \frac{1}{2} g^{t t}< \derm_t \Lie_{Z^I}  \Phi_{V} ,   \derm_t \Lie_{Z^I}  \Phi_{V} >   + \frac{1}{2} g^{tj}< \derm_j \Lie_{Z^I}  \Phi_{V} ,   \derm_t \Lie_{Z^I}  \Phi_{V}>  \\
\notag
&&- \frac{1}{2} g^{j t}  < \derm_j \Lie_{Z^I}  \Phi_{V} ,   \derm_t \Lie_{Z^I}  \Phi_{V} >   - \frac{1}{2} g^{j i}  < \derm_j \Lie_{Z^I}  \Phi_{V} ,   \derm_i \Lie_{Z^I}  \Phi_{V}>   \; . \\
 \eea
Consequently, the conservation law \ref{conservationlawforvectorfieldt}, with the vector field $X= \frac{\pa}{ \pa t}$\,, gives the stated result.
 \end{proof}
 
 \begin{corollary}
 We have
    \bea
  \notag
   &&   \int_{\Sigma^{ext}_{t_2} } \Big(  - \frac{1}{2} g^{t t}< \derm_t \Lie_{Z^I}  \Phi_{V} ,   \derm_t \Lie_{Z^I}  \Phi_{V} >   + \frac{1}{2} g^{j i}  < \derm_j \Lie_{Z^I}  \Phi_{V} ,   \derm_i \Lie_{Z^I}  \Phi_{V} > \Big)    \cdot d^{n}x \\
   \notag
&& +  \int_{N_{t_1}^{t_2} }  \big( T_{\hat{L} t}^{(\bf{g})} \big)  \cdot dv^{(\bf{m})}_N \\
     \notag
   &=&   \int_{\Sigma^{ext}_{t_1} }   \Big(  - \frac{1}{2} g^{t t}< \derm_t \Lie_{Z^I}  \Phi_{V} ,   \derm_t \Lie_{Z^I}  \Phi_{V} >   + \frac{1}{2} g^{j i}  < \derm_j \Lie_{Z^I}  \Phi_{V} ,   \derm_i \Lie_{Z^I}  \Phi_{V}>  \Big)  \cdot d^{n}x  \\
     \notag
     &&- \int_{t_1}^{t_2}  \int_{\Sigma^{ext}_{\tau} } \Big(  < g^{\mu\a} \derm_{\mu } \derm_\a \Lie_{Z^I}  \Phi_{V} ,   \derm_t \Lie_{Z^I}  \Phi_{V}>  \\
       \notag
&& +  \frac{1}{2} \cdot   ( \derm_{t } g^{t \a} ) \cdot < \derm_\a \Lie_{Z^I}  \Phi_{V} ,   \derm_t \Lie_{Z^I}  \Phi_{V} > +( \derm_{j } g^{j \a} ) \cdot < \derm_\a \Lie_{Z^I}  \Phi_{V} ,   \derm_t \Lie_{Z^I}  \Phi_{V} > \\
&& - \frac{1}{2} \cdot  ( \derm_{t } g^{j \b} ) \cdot < \derm_j \Lie_{Z^I}  \Phi_{V} ,   \derm_\b \Lie_{Z^I}  \Phi_{V} >     \Big) \cdot d\tau d^{n}x   \; .
 \eea
 
 \end{corollary}
 
 \begin{proof}
 We compute
  \bea\label{expressionofdivergenceofTmutwithoutdecompisinginaframe}
  \notag
 && \derm^{\mu }T_{\mu t}^{(\bf{g})}  \\
 \notag
  &=& < g^{\mu\a} \derm_{\mu } \derm_\a \Lie_{Z^I}  \Phi_{V} ,   \derm_t \Lie_{Z^I}  \Phi_{V} >  +( \derm_{\mu } g^{\mu\a} ) \cdot < \derm_\a \Lie_{Z^I}  \Phi_{V} ,   \derm_t \Lie_{Z^I}  \Phi_{V}> \\
&&- \frac{1}{2} m^{\mu}_{\;\;\; t}  \cdot  ( \derm_{\mu } g^{\a\b} ) \cdot < \derm_\a \Lie_{Z^I}  \Phi_{V},   \derm_\b\Lie_{Z^I}  \Phi_{V}>  \; .
 \eea
Since $ m^{\mu}_{\;\;\; t} =  - m^{\mu t} $, we compute further by decomposing the sum in wave coordinates,
\bea\label{evaluationofthecovariantdivergenceofourenergymomentumtensor}
\notag
 && \derm^{\mu }T_{\mu t}^{(\bf{g})}    \\
 \notag
 &=& < g^{\mu\a} \derm_{\mu } \derm_\a \Lie_{Z^I}  \Phi_{V} ,   \derm_t \Lie_{Z^I}  \Phi_{V}>  \\
\notag
&& +  \frac{1}{2} \cdot   ( \derm_{t } g^{t \a} ) \cdot < \derm_\a \Lie_{Z^I}  \Phi_{V} ,   \derm_t \Lie_{Z^I}  \Phi_{V} > +( \derm_{j } g^{j \a} ) \cdot < \derm_\a \Lie_{Z^I}  \Phi_{V} ,   \derm_t \Lie_{Z^I}  \Phi_{V} > \\
&& - \frac{1}{2} \cdot  ( \derm_{t } g^{j \b} ) \cdot < \derm_j \Lie_{Z^I}  \Phi_{V},   \derm_\b \Lie_{Z^I}  \Phi_{V}> \, .
\eea
 Inserting in Lemma \ref{conservationlawwithoutcomputingdivergenceofTmut}, we obtain the desired result.
 \end{proof}

 \subsection{The weighted energy estimate in the exterior for $g^{\mu\nu} \derm_{\mu} \derm_{\nu} \Phi$ }\

\begin{definition}\label{definitionoftheparmeterq}
We define $q := r - t $\,, where $t$ and $r$ are defined using the wave coordinates as explained in \cite{G4}.
\end{definition}

\begin{definition}\label{defoftheweightw}
We define for some $\gamma > 0$\,,
\beaa
w(q):=\begin{cases} (1+|q|)^{1+2\gamma} \quad\text{when }\quad q>0 , \\
         1 \,\quad\text{when }\quad q<0 . \end{cases}
\eeaa

\end{definition}

\begin{definition}\label{defwidehatw}
We define $\widehat{w}$ for $\ga > 0$ and $\mu < 0$\,, by 
\beaa
\widehat{w}(q)&:=&\begin{cases} (1+|q|)^{1+2\gamma} \quad\text{when }\quad q>0 , \\
        (1+|q|)^{2\mu}  \,\quad\text{when }\quad q<0 . \end{cases} 
\eeaa
Note that the factor $\mu \neq 0$ is constructed so that for $q<0$, the derivative $\frac{\pa \widehat{w}}{\pa q}$ is non-vanishing, so as to have a control on certain tangential derivatives, which is needed. This being said, note that the definition of $\widehat{w}$\,, is also so that for $\ga \neq - \frac{1}{2} $ and $\mu \neq 0$ (which is assumed here), we would have 
\beaa
\widehat{w}^{\prime}(q) \sim \frac{\widehat{w}(q)}{(1+|q|)} \; ,
\eeaa
(see Lemma \ref{derivativeoftildwandrelationtotildew}) -- this is will determine the kind of control that we will have on the tangential derivatives. 

\end{definition}

\begin{remark}
We take $\mu < 0$ (instead of $\mu > 0$), because we want the derivative $\frac{\pa \widehat{w}}{\pa q} > 0$\,, as we will see that this is what we need in order to obtain an energy estimate on the fields (see Corollary \ref{TheenerhyestimatewithtermsinvolvingHandderivativeofHandwithwandhatw}). In other words, $\mu < 0$ is a necessary condition to ensure that $\widehat{w}^{\prime} (q)$ enters with the right sign in the energy estimate.
\end{remark}

\begin{definition}\label{defwidetildew}
We define $\widetilde{w}$ by 
\beaa
\widetilde{w} ( q)&:=&  \widehat{w}(q) + w(q) :=\begin{cases} 2 (1+|q|)^{1+2\gamma} \quad\text{when }\quad q>0 \, , \\
       1+  (1+|q|)^{2\mu}  \,\quad\text{when }\quad q<0 \, . \end{cases} \\    
\eeaa
Note that the definition of $\widetilde{w}$ is constructed so that Lemma \ref{equaivalenceoftildewandtildeandofderivativeoftildwandderivativeofhatw} holds.
\end{definition}

\begin{lemma}\label{equaivalenceoftildewandtildeandofderivativeoftildwandderivativeofhatw}
We have
\beaa
\widetilde{w}^{\prime}  &\sim & \widehat{w}^{\prime } \; .
\eeaa
Furthermore, for $\mu < 0$\,, we have
\beaa
\widetilde{w} ( q)& \sim & w(q) \; .
\eeaa

\end{lemma}

\begin{proof}
We compute the derivative with respect to $q$\,,
\beaa
\widetilde{w}^{\prime} &=&   \widehat{w}^{\prime} ( q) + w^{\prime} ( q) = \begin{cases} 2 \cdot \widehat{w}^{\prime} ( q) \quad\text{when }\quad q>0 \, , \\
        \widehat{w}^{\prime }( q)  \,\quad\text{when }\quad q<0 \,. \end{cases} \\ 
\eeaa
Consequently, $\widetilde{w}^{\prime}  \sim  \widehat{w}^{\prime } $\;. Now, on one hand, since $ \widehat{w} \geq 0$, we have 
\beaa
\widetilde{w} ( q)&\geq&w(q) \;.
\eeaa
On the other hand, since $\mu < 0$\,, we have
\beaa
\widetilde{w} ( q)& = & \begin{cases} 2 (1+|q|)^{1+2\gamma} \quad\text{when }\quad q>0\, , \\
       1+  (1+|q|)^{2\mu}  \,\quad\text{when }\quad q<0 \,. \end{cases} \leq 2 \cdot w(q) \;.
\eeaa
Thus, $\widetilde{w} ( q) \sim w(q) $\;.
\end{proof}

\begin{lemma}\label{derivativeoftildwandrelationtotildew}
Let $\widehat{w}$ be defined as in Definition \ref{defwidehatw}.
We have, for $\ga \neq - \frac{1}{2} $ and $\mu \neq 0$\,,  
\beaa
\widehat{w}^{\prime}(q) \sim \frac{\widehat{w}(q)}{(1+|q|)} \; .
\eeaa

\end{lemma}
\begin{proof}

We compute,
\beaa
\widehat{w}^{\prime}(q)&=&\begin{cases} (1+2\gamma)(1+|q|)^{2\gamma} \quad\text{when }\quad q>0 \, , \\
         - 2\mu(1+|q|)^{2\mu-1} \,\quad\text{when }\quad q<0 \,. \end{cases} =\begin{cases}         (1+2\gamma)   \frac{w(q)}{(1+|q|)}   \quad\text{when }\quad q>0\, , \\
         -  2\mu      \frac{w(q)}{(1+|q|)}    \,\quad\text{when }\quad q<0 \,. \end{cases} 
\eeaa
Thus,
\beaa
\min \{ (1+2\gamma), -2\mu\}  \cdot \frac{\widehat{w}(q)}{(1+|q|)} 
 \leq \widehat{w}^{\prime}(q) \leq \max \{ (1+2\gamma), -2\mu\} \cdot \frac{\widehat{w}(q)}{(1+|q|)} \; ,
\eeaa
and hence, for  $\min \{ (1+2\gamma), -2\mu\}  \neq 0 $ and $\max \{ (1+2\gamma), -2\mu\} \neq 0$\,, we have the result.
\end{proof}

We now establish a conservation law with the weight $\widetilde{w}$\,. 

   \begin{lemma}\label{weightedconservationlawintheexteriorwiththeenergymomuntumtensorcontarctedwithvectordt}
We have
    \beaa
      \notag
  && \int_{N_{t_1}^{t_2} }  \big( T_{\hat{L} t}^{(\bf{g})} \big)  \cdot  \widetilde{w} (q) \cdot dv^{(\bf{m})}_N  +   \int_{\Sigma^{ext}_{t_2} }   T_{t t}^{(\bf{g})}  \cdot \widetilde{w} (q) \cdot d^{n}x \\
   \notag
     &=&   \int_{\Sigma^{ext}_{t_1} }  T_{ t t}^{(\bf{g})}   \cdot \widetilde{w} (q)  \cdot d^{n}x  -  \int_{t_1}^{t_2}  \int_{\Sigma^{ext}_{\tau} }  \Big( T_{t  t}^{(\bf{g})} +   T_{r  t}^{(\bf{g})} \Big) \cdot d\tau \cdot \widetilde{w} ^\prime (q)  \cdot  d^{n}x \\
     \notag
   &&  -  \int_{t_1}^{t_2}  \int_{\Sigma^{ext}_{\tau} } \Big(   \derm^{\mu }T_{\mu t}^{(\bf{g})}    \Big) \cdot d\tau \cdot \widetilde{w} (q) \cdot d^{n}x \; .  
 \eeaa
 \end{lemma}
 
 \begin{proof}
 Considering again the Minkowski metric $m$ in the coordinates $\{t, x^1, x^2, x^3 \}$, and instead of contracting $T_{ \mu\nu}^{(\bf{g})}$  with $\frac{\pa}{\pa t} $ with respect to the second component $\nu$\,, we contract with the weighted vector
 \bea\label{The weightedvectorproportionaltodt}
 X = \widetilde{w} (q) \frac{\pa}{\pa t} \;.
 \eea
 By then, we have in the coordinates $\mu, \nu \in \{t, x^1, x^2, x^3 \}$, 
  \beaa
 \big( \derm^{\mu}  \big( \widetilde{w} (q) \frac{\pa}{\pa t}  \big)^{\nu} \big) &=&  \widetilde{w} ^\prime(q) \cdot m^{\mu\a} \derm_{\a} (q) \cdot  \big( \frac{\pa}{\pa t}  \big)^{\nu} \; .
 \eeaa
Since $q = r-t$\,, for $\mu = t = x^0$, we have $m^{\mu\a} \derm_{\a} (q) =  1 \; ,$ and for $\mu = x^j$, we have $m^{\mu\a} \derm_{\a} (q) =   \frac{x^{j}}{r}\; .$ Thus,
 \beaa
 \notag
 \big( \derm^{\mu} \big( \widetilde{w} (q)  \frac{\pa}{\pa t}  \big)^{\nu} \big) \cdot T_{\mu\nu}^{(\bf{g})}   &=&   \widetilde{w} ^\prime(q) \cdot \Big( T_{t  t}^{(\bf{g})} +   T_{r  t}^{(\bf{g})} \Big) \; .
 \eeaa
We still have $n^{(\bf{m}), \nu}_{\Sigma}  =  \big( \frac{\pa}{\pa t}  \big)^{\nu} \; ,$ \,$ dv^{(\bf{m})}_\Sigma= dx^1 dx^2 dx^3  := d^{3} x \; ,  $ and $ n^{(\bf{m}), \nu}_{N} = \hat{L}^{\nu} \; .$ Consequently, the conservation law with the weighted vector $w(q) \frac{\pa}{\pa t}$ contracted with the second component of the non-symmetric tensor $T_{\mu\nu}^{(\bf{g})}$, gives the following equality
 \beaa
 \notag
  && \int_{t_1}^{t_2}  \int_{\Sigma^{ext}_{\tau} }  \Big( T_{t  t}^{(\bf{g})} +   T_{r  t}^{(\bf{g})} \Big) \cdot d\tau \cdot \widetilde{w}^\prime (q)\cdot  d^{n}x +  \int_{t_1}^{t_2}  \int_{\Sigma^{ext}_{\tau} } \Big(   \derm^{\mu }T_{\mu t}^{(\bf{g})}    \Big) \cdot d\tau \cdot \widetilde{w}(q) \cdot  d^{n}x \\
   \notag
     &=&   \int_{\Sigma^{ext}_{t_1} }  T_{ t t}^{(\bf{g})}   \cdot \widetilde{w}(q)  \cdot d^{n}x -  \int_{\Sigma^{ext}_{t_2} (q)}   T_{t t}^{(\bf{g})}  \cdot \widetilde{w}(q) \cdot d^{n}x  - \int_{N_{t_1}^{t_2} }  \big( T_{\hat{L} t}^{(\bf{g})} \big)  \cdot  \widetilde{w}(q) \cdot dv^{(\bf{m})}_N  \; .
 \eeaa
 \end{proof}

 \begin{corollary}\label{WeightedenergyestimateusingHandnosmallnessyet}
 We have
    \beaa 
     \notag
  && \int_{N_{t_1}^{t_2} }  \big( T_{\hat{L} t}^{(\bf{g})} \big)  \cdot  \widetilde{w}(q) \cdot dv^{(\bf{m})}_N  +   \int_{\Sigma^{ext}_{t_2} }  \big(    - \frac{1}{2} (m^{t t} + H^{tt} ) < \derm_t  \Lie_{Z^I}  \Phi_{V} ,   \derm_t  \Lie_{Z^I}  \Phi_{V} >   \\
  &&\quad\quad\quad \quad  + \frac{1}{2} (m^{j i} +H^{j i} )   < \derm_j  \Lie_{Z^I}  \Phi_{V} ,   \derm_i  \Lie_{Z^I}  \Phi_{V} > \big)  \cdot \widetilde{w}(q) \cdot d^{n}x \\
   \notag
     &=&   \int_{\Sigma^{ext}_{t_1} } \big(    - \frac{1}{2} (m^{t t} + H^{tt} ) < \derm_t  \Lie_{Z^I}  \Phi_{V},   \derm_t  \Lie_{Z^I}  \Phi_{V} >    \\
  &&\quad\quad\quad  + \frac{1}{2} (m^{j i} +H^{j i} )   < \derm_j  \Lie_{Z^I}  \Phi_{V} ,   \derm_i  \Lie_{Z^I}  \Phi_{V} > \big)   \cdot \widetilde{w}(q)  \cdot d^{n}x  \\
     && -  \int_{t_1}^{t_2}  \int_{\Sigma^{ext}_{\tau} }  \Big( T_{t  t}^{(\bf{g})} +   T_{r  t}^{(\bf{g})} \Big) \cdot d\tau \cdot \widetilde{w}^\prime (q) d^{n}x \\
     \notag
   &&  -  \int_{t_1}^{t_2}  \int_{\Sigma^{ext}_{\tau} } \Big(  < g^{\mu\a} \derm_{\mu } \derm_\a  \Lie_{Z^I}  \Phi_{V},   \derm_t  \Lie_{Z^I}  \Phi_{V} >  \\
   \notag
   && +( \derm_{\mu } H^{\mu\a} ) \cdot < \derm_\a  \Lie_{Z^I}  \Phi_{V} ,   \derm_t  \Lie_{Z^I}  \Phi_{V} > \\
 \notag
&&- \frac{1}{2} m^{\mu}_{\;\;\; t}  \cdot  ( \derm_{\mu } H^{\a\b} ) \cdot < \derm_\a  \Lie_{Z^I}  \Phi_{V} ,   \derm_\b  \Lie_{Z^I}  \Phi_{V}>   \Big) \cdot d\tau \cdot \widetilde{w}(q) d^{n}x \; .  
 \eeaa
 \end{corollary}
 
 \begin{proof}
We want to evaluate the terms in \eqref{weightedconservationlawintheexteriorwiththeenergymomuntumtensorcontarctedwithvectordt}. We have shown based on \eqref{evlautingTttforourenergymomentumtensor}, that
   \beaa
 \notag
  T_{ t t}^{(\bf{g})} &=&    - \frac{1}{2} (m^{t t} + H^{tt} ) < \derm_t \Lie_{Z^I}  \Phi_{V},   \derm_t  \Lie_{Z^I}  \Phi_{V}>   + \frac{1}{2} (m^{j i} +H^{j i} )   < \derm_j  \Lie_{Z^I}  \Phi_{V} ,   \derm_i  \Lie_{Z^I}  \Phi_{V}> \; ,
  \eeaa
and based on \eqref{expressionofdivergenceofTmutwithoutdecompisinginaframe}, that
 \beaa
  \notag
 \derm^{\mu }T_{\mu t}^{(\bf{g})}    &=& < g^{\mu\a} \derm_{\mu } \derm_\a  \Lie_{Z^I}  \Phi_{V},   \derm_t  \Lie_{Z^I}  \Phi_{V} >  +( \derm_{\mu } H^{\mu\a} ) \cdot < \derm_\a  \Lie_{Z^I}  \Phi_{V} ,   \derm_t  \Lie_{Z^I}  \Phi_{V} >  \\
 && - \frac{1}{2} m^{\mu}_{\;\;\; t}  \cdot  ( \derm_{\mu } H^{\a\b} ) \cdot < \derm_\a  \Lie_{Z^I}  \Phi_{V},   \derm_\b  \Lie_{Z^I}  \Phi_{V}> \; . 
\eeaa
Inserting these in Lemma \eqref{weightedconservationlawintheexteriorwiththeenergymomuntumtensorcontarctedwithvectordt}, we get the stated result.

  \end{proof}

Now, we would like to evaluate in Corollary \ref{WeightedenergyestimateusingHandnosmallnessyet} the term with the weight $w^\prime (q)$\,.

 \begin{lemma}\label{decompositionofthegradientintotangentialandnontangentialpart}
 In our fixed system of coordinates, for $i\,, j $ running overs spatial indices, we have
 \bea
\notag
 && \de^{ij}  < \derm_i  \Lie_{Z^I}  \Phi_{V} , \derm_j  \Lie_{Z^I}  \Phi_{V} >  \\
 &=& \de^{ij}  < ( \derm_i - \frac{x_i}{r} \derm_{r}  ) \Lie_{Z^I}  \Phi_{V}, (\derm_j - \frac{x_j}{r} \derm_{r} )  \Lie_{Z^I}  \Phi_{V} > +  | \derm_{r}  \Lie_{Z^I}  \Phi_{V}|^2   \; . 
\eea
Furthermore, 
 \bea
\notag
&&  | \derm_t   \Lie_{Z^I}  \Phi_{V}+ \derm_r  \Lie_{Z^I}  \Phi_{V} |^2       +  \de^{ij}  < ( \derm_i - \frac{x_i}{r} \derm_{r}  ) \Lie_{Z^I}  \Phi_{V} , (\derm_j - \frac{x_j}{r} \derm_{r} )  \Lie_{Z^I}  \Phi_{V}>  \\
 &=& |\derm  \Lie_{Z^I}  \Phi_{V} |^2  + 2  < \derm_t   \Lie_{Z^I}  \Phi_{V}, \derm_r  \Lie_{Z^I}  \Phi_{V} > \; .
   \eea
\end{lemma}
 
 \begin{proof}
 
We have
\beaa
\derm_r  \Lie_{Z^I}  \Phi_{V}= \frac{x^j}{r} \cdot  \derm_j   \Lie_{Z^I}  \Phi_{V} \;.
\eeaa 
We consider the derivatives restricted on the $3$-spheres, that is $  \pa_i - \frac{x_i}{r} \pa_{r} $\,. We have
\beaa
&& \de^{ij}  < ( \derm_i - \frac{x_i}{r} \derm_{r}  ) \Lie_{Z^I}  \Phi_{V}, (\derm_j - \frac{x_j}{r} \derm_{r} )  \Lie_{Z^I}  \Phi_{V}>\\
&=& \de^{ij}  < \derm_i  \Lie_{Z^I}  \Phi_{V}, \derm_j  \Lie_{Z^I}  \Phi_{V}> - 2 |  \derm_{r}  \Lie_{Z^I}  \Phi_{V} |^2   +  | \derm_{r}  \Lie_{Z^I}  \Phi_{V} |^2   \;.
\eeaa
Hence,
\beaa
\notag
\de^{ij}  < \derm_i  \Lie_{Z^I}  \Phi_{V}, \derm_j  \Lie_{Z^I}  \Phi_{V}> &=& \de^{ij}  < ( \derm_i - \frac{x_i}{r} \derm_{r}  ) \Lie_{Z^I}  \Phi_{V} , (\derm_j - \frac{x_j}{r} \derm_{r} )  \Lie_{Z^I}  \Phi_{V} > \\
&& +  | \derm_{r}   \Lie_{Z^I}  \Phi_{V} |^2 \;.
\eeaa

Since $\derm$ is the Minkowski covariant derivative, computing the trace with respect to our fixed system of coordinates $\{t, x^1, x^2, x^3 \}$, we get
\beaa
\notag
 && \de^{ij}  < \derm_i  \Lie_{Z^I}  \Phi_{V} , \derm_j  \Lie_{Z^I}  \Phi_{V} >  \\
 \notag
 &=& \de^{ij}  < ( \derm_i - \frac{x_i}{r} \derm_{r}  ) \Lie_{Z^I}  \Phi_{V} , (\derm_j - \frac{x_j}{r} \derm_{r} )  \Lie_{Z^I}  \Phi_{V}> +  < \derm_{r}   \Lie_{Z^I}  \Phi_{V}, \derm_{r}  \Lie_{Z^I}  \Phi_{V} > \; . 
\eeaa

Therefore, 
 \beaa
\notag
&&  | \derm_t   \Lie_{Z^I}  \Phi_{V}+ \derm_r  \Lie_{Z^I}  \Phi_{V}|^2    +  \de^{ij}  < ( \derm_i - \frac{x_i}{r} \derm_{r}  ) \Lie_{Z^I}  \Phi_{V}, (\derm_j - \frac{x_j}{r} \derm_{r} )  \Lie_{Z^I}  \Phi_{V} >  \\
\notag
 &=&  | \derm_t   \Lie_{Z^I}  \Phi_{V} |^2     + 2  < \derm_t   \Lie_{Z^I}  \Phi_{V}, \derm_r   \Lie_{Z^I}  \Phi_{V} > + \de^{ij}  < \derm_i  \Lie_{Z^I}  \Phi_{V} , \derm_j  \Lie_{Z^I}  \Phi_{V}> \\
 &=& |\derm  \Lie_{Z^I}  \Phi_{V} |^2  + 2  < \derm_t   \Lie_{Z^I}  \Phi_{V}, \derm_r   \Lie_{Z^I}  \Phi_{V} > \; .
   \eeaa
   
 \end{proof}

 \begin{lemma}\label{expressionofTttplusTtr}
 We have
    \beaa
   \notag
 T_{ t t}^{(\bf{g})}  +  T_{ r t}^{(\bf{g})}  &=&  \frac{1}{2} \Big(  | \derm_t   \Lie_{Z^I}  \Phi_{V} + \derm_r  \Lie_{Z^I}  \Phi_{V}  |^2  +  \sum_{i=1}^n  | \derm_i - \frac{x_i}{r} \derm_{r}  )  \Lie_{Z^I}  \Phi_{V}  |^2 \Big)   \\
     \notag
 &&  - \frac{1}{2} H^{t t} | \derm_t  \Lie_{Z^I}  \Phi_{V} |^2   + \frac{1}{2} H^{j i}  < \derm_j  \Lie_{Z^I}  \Phi_{V}  ,   \derm_i  \Lie_{Z^I}  \Phi_{V} >  \\
    \notag
&&  + H^{r t} | \derm_t  \Lie_{Z^I}  \Phi_{V} |^2 + H^{r j}< \derm_j  \Lie_{Z^I}  \Phi_{V},   \derm_t  \Lie_{Z^I}  \Phi_{V} >  \; . 
  \eeaa
  
 \end{lemma}

   \begin{proof}
   
 We compute
  \beaa
T_{ r t}^{(\bf{g})}  &=& {T^{(\bf{g})}}^{r}_{\;\;\; t}  =  g^{r\a}< \derm_\a  \Lie_{Z^I}  \Phi_{V},   \derm_t  \Lie_{Z^I}  \Phi_{V} > \;  \\
&=& g^{r t}< \derm_t  \Lie_{Z^I}  \Phi_{V},   \derm_t  \Lie_{Z^I}  \Phi_{V}> + g^{r j}< \derm_j  \Lie_{Z^I}  \Phi_{V},   \derm_t  \Lie_{Z^I}  \Phi_{V}>   \; . 
 \eeaa
  Consequently,
      \beaa
      \notag
 T_{ t t}^{(\bf{g})}  +  T_{ r t}^{(\bf{g})} &=& - \frac{1}{2} m^{t t}< \derm_t  \Lie_{Z^I}  \Phi_{V},   \derm_t  \Lie_{Z^I}  \Phi_{V}>   + \frac{1}{2} m^{j i}  < \derm_j  \Lie_{Z^I}  \Phi_{V} ,   \derm_i  \Lie_{Z^I}  \Phi_{V} >  \\
&& + m^{r t}< \derm_t  \Lie_{Z^I}  \Phi_{V},   \derm_t  \Lie_{Z^I}  \Phi_{V}> + m^{r j}< \derm_j  \Lie_{Z^I}  \Phi_{V} ,   \derm_t  \Lie_{Z^I}  \Phi_{V}>  \\
&& - \frac{1}{2} H^{t t}< \derm_t  \Lie_{Z^I}  \Phi_{V},   \derm_t  \Lie_{Z^I}  \Phi_{V}>   + \frac{1}{2} H^{j i}  < \derm_j  \Lie_{Z^I}  \Phi_{V} ,   \derm_i  \Lie_{Z^I}  \Phi_{V}>  \\
&&  + H^{r t}< \derm_t  \Lie_{Z^I}  \Phi_{V} ,   \derm_t  \Lie_{Z^I}  \Phi_{V}> + H^{r j}< \derm_j  \Lie_{Z^I}  \Phi_{V} ,   \derm_t  \Lie_{Z^I}  \Phi_{V} >   \; .
  \eeaa
 Thus,
       \beaa
      \notag
 && T_{ t t}^{(\bf{g})}  +  T_{r t}^{(\bf{g})} \\
  &=&  \frac{1}{2} | \derm_t  \Lie_{Z^I}  \Phi_{V} |^2   + \frac{1}{2} \de^{j i}  < \derm_j  \Lie_{Z^I}  \Phi_{V} ,   \derm_i  \Lie_{Z^I}  \Phi_{V} > \\
  &&  + m^{r j}< \derm_j  \Lie_{Z^I}  \Phi_{V} ,   \derm_t  \Lie_{Z^I}  \Phi_{V}>   - \frac{1}{2} H^{t t}< \derm_t  \Lie_{Z^I}  \Phi_{V},   \derm_t  \Lie_{Z^I}  \Phi_{V}>  \\
  && + \frac{1}{2} H^{j i}  < \derm_j  \Lie_{Z^I}  \Phi_{V},   \derm_i  \Lie_{Z^I}  \Phi_{V}>   + H^{r t}< \derm_t  \Lie_{Z^I}  \Phi_{V},   \derm_t  \Lie_{Z^I}  \Phi_{V}> \\
  && + H^{r j}< \derm_j  \Lie_{Z^I}  \Phi_{V} ,   \derm_t  \Lie_{Z^I}  \Phi_{V} > \; .
\eeaa
 Yet, $m^{r j}  = \frac{x^j}{r} \;,$ hence, 
   \beaa
   m^{r j}< \derm_j  \Lie_{Z^I}  \Phi_{V} ,   \derm_t  \Lie_{Z^I}  \Phi_{V} > &=& \frac{x^j}{r} < \derm_j  \Lie_{Z^I}  \Phi_{V},   \derm_t  \Lie_{Z^I}  \Phi_{V}> \\
   &=& < \derm_r  \Lie_{Z^I}  \Phi_{V},   \derm_t  \Lie_{Z^I}  \Phi_{V}> \; .
   \eeaa
Therefore, we have
\beaa
\notag
 T_{ t t}^{(\bf{g})}  +  T_{  r t}^{(\bf{g})} &=&  \frac{1}{2} | \derm  \Lie_{Z^I}  \Phi_{V} |^2  + < \derm_r \Lie_{Z^I}  \Phi_{V},   \derm_t  \Lie_{Z^I}  \Phi_{V}>  \\
 \notag
&& - \frac{1}{2} H^{t t}< \derm_t  \Lie_{Z^I}  \Phi_{V} ,   \derm_t  \Lie_{Z^I}  \Phi_{V} >   + \frac{1}{2} H^{j i}  < \derm_j  \Lie_{Z^I}  \Phi_{V} ,   \derm_i  \Lie_{Z^I}  \Phi_{V}>  \\
\notag
&&  + H^{r t}< \derm_t  \Lie_{Z^I}  \Phi_{V},   \derm_t  \Lie_{Z^I}  \Phi_{V} > + H^{r j}< \derm_j  \Lie_{Z^I}  \Phi_{V},   \derm_t  \Lie_{Z^I}  \Phi_{V} > \; . 
  \eeaa

Therefore, based on Lemma \ref{decompositionofthegradientintotangentialandnontangentialpart},
   \beaa
   \notag
 T_{ t t}^{(\bf{g})}  +  T_{ r t}^{(\bf{g})} &=&  \frac{1}{2} \Big(  < \derm_t   \Lie_{Z^I}  \Phi_{V} + \derm_r \ \Lie_{Z^I}  \Phi_{V}, \derm_t   \Lie_{Z^I}  \Phi_{V} + \derm_r  \Lie_{Z^I}  \Phi_{V} >    \\
     \notag
  && +  \de^{ij}  < ( \derm_i - \frac{x_i}{r} \derm_{r}  )\Phi , (\derm_j - \frac{x_j}{r} \derm_{r} )  \Lie_{Z^I}  \Phi_{V}> \Big)   \\
     \notag
 &&  - \frac{1}{2} H^{t t} | \derm_t  \Lie_{Z^I}  \Phi_{V} |^2   + \frac{1}{2} H^{j i}  < \derm_j  \Lie_{Z^I}  \Phi_{V} ,   \derm_i  \Lie_{Z^I}  \Phi_{V}>  \\
    \notag
&&  + H^{r t} |  \derm_t  \Lie_{Z^I}  \Phi_{V} |^2 + H^{r j}< \derm_j  \Lie_{Z^I}  \Phi_{V} ,   \derm_t  \Lie_{Z^I}  \Phi_{V} > \; . 
  \eeaa
  \end{proof}

\begin{lemma}\label{comutingthetermthatcarriesthederivativeoftheweight}
We have
       \beaa
 && T_{ t t}^{(\bf{g})}  +  T_{  r t}^{(\bf{g})} \\
 &=&  \frac{1}{2} \Big(  |\derm_t   \Lie_{Z^I}  \Phi_{V} + \derm_r  \Lie_{Z^I}  \Phi_{V} |^2     +  \de^{ij}  < ( \derm_i - \frac{x_i}{r} \derm_{r}  ) \Lie_{Z^I}  \Phi_{V}, (\derm_j - \frac{x_j}{r} \derm_{r} )  \Lie_{Z^I}  \Phi_{V}> \Big)   \\
  && -2 H^{\underline{L}\a}< \derm_\a  \Lie_{Z^I}  \Phi_{V},   \derm_t  \Lie_{Z^I}  \Phi_{V} >  + \frac{1}{2}   \cdot H^{\a\b}  < \derm_\a  \Lie_{Z^I}  \Phi_{UV} ,   \derm_\b  \Lie_{Z^I}  \Phi_{V} >  \,.
 \eeaa

  \end{lemma}

\begin{proof}

We have
\bea
\notag
 T_{ t t}^{(\bf{g})}  +  T_{  r t}^{(\bf{g})} &=&T_{  L t}^{(\bf{g})}  = m_{L \underline{L}} \cdot T^{(\bf{g}) \; \underline{L}}_{\;\;\;\;\;\;\;\;\;\; t}   \; .
\eea
  Since $m_{L \underline{L}}  = -2$\,, we get
    \beaa
 T_{ t t}^{(\bf{g})}  +  T_{  r t}^{(\bf{g})} &=&  - 2 \cdot T^{(\bf{g}) \; \underline{L}}_{\;\;\;\;\;\;\;\;\;\; t}   \; .
  \eeaa

Given the expression of the energy momentum-tensor, we obtain

  \beaa
\notag
T^{(\bf{g}) \; \underline{L}}_{\;\;\;\;\;\;\;\;\;\; t}   =  g^{\underline{L}\a}< \derm_\a  \Lie_{Z^I}  \Phi_{V},   \derm_t  \Lie_{Z^I}  \Phi_{V} > - \frac{1}{2} m^{\underline{L}}_{\;\;\; t}  \cdot g^{\a\b}  < \derm_\a  \Lie_{Z^I}  \Phi_{V},   \derm_\b  \Lie_{Z^I}  \Phi_{V}> \;, 
 \eeaa
  
We compute
  \beaa
  m^{\underline{L}}_{\;\;\; t}  = m^{\underline{L} \mu} \cdot m_{\mu t } =  m^{\underline{L} L } \cdot m_{L t } = m^{\underline{L} L } \cdot m_{t t } = \frac{1}{2} \; .
  \eeaa

Hence,
  \beaa
\notag
T^{(\bf{g}) \; \underline{L}}_{\;\;\;\;\;\;\;\;\;\; t}   &=&  g^{\underline{L}\a}< \derm_\a  \Lie_{Z^I}  \Phi_{V},   \derm_t  \Lie_{Z^I}  \Phi_{V} > - \frac{1}{4}   \cdot g^{\a\b}  < \derm_\a  \Lie_{Z^I}  \Phi_{V},   \derm_\b  \Lie_{Z^I}  \Phi_{V} > \\
 &=&  m^{\underline{L}\a}< \derm_\a  \Lie_{Z^I}  \Phi_{V},   \derm_t  \Lie_{Z^I}  \Phi_{V} > - \frac{1}{4}   \cdot m^{\a\b}  < \derm_\a  \Lie_{Z^I}  \Phi_{V},   \derm_\b  \Lie_{Z^I}  \Phi_{V} >  \\
  && +  H^{\underline{L}\a}< \derm_\a  \Lie_{Z^I}  \Phi_{V},   \derm_t  \Lie_{Z^I}  \Phi_{V} > - \frac{1}{4}   \cdot H^{\a\b}  < \derm_\a  \Lie_{Z^I}  \Phi_{V} ,   \derm_\b  \Lie_{Z^I}  \Phi_{V}>  \; ,
 \eeaa
 and therefore,
     \beaa
 T_{ t t}^{(\bf{g})}  +  T_{  r t}^{(\bf{g})} &=&  - 2 \cdot T^{(\bf{g}) \; \underline{L}}_{\;\;\;\;\;\;\;\;\;\; t}   \\
 &=& -2 m^{\underline{L}\a}< \derm_\a  \Lie_{Z^I}  \Phi_{V},   \derm_t  \Lie_{Z^I}  \Phi_{V} > + \frac{1}{2}   \cdot m^{\a\b}  < \derm_\a  \Lie_{Z^I}  \Phi_{V} ,   \derm_\b  \Lie_{Z^I}  \Phi_{V} >  \\
  && -2 H^{\underline{L}\a}< \derm_\a  \Lie_{Z^I}  \Phi_{V},   \derm_t  \Lie_{Z^I}  \Phi_{V} >  + \frac{1}{2}   \cdot H^{\a\b}  < \derm_\a  \Lie_{Z^I}  \Phi_{V},   \derm_\b  \Lie_{Z^I}  \Phi_{V}>  \; .
 \eeaa

However, we already computed the part with the Minkowski metric $m$ and therefore, we have
 
 \beaa
 && -2 m^{\underline{L}\a}< \derm_\a  \Lie_{Z^I}  \Phi_{V},   \derm_t  \Lie_{Z^I}  \Phi_{V}> + \frac{1}{2}   \cdot m^{\a\b}  < \derm_\a  \Lie_{Z^I}  \Phi_{V},   \derm_\b  \Lie_{Z^I}  \Phi_{V} >  \\
  &=&  \frac{1}{2} \Big(  < \derm_t   \Lie_{Z^I}  \Phi_{V} + \derm_r  \Lie_{Z^I}  \Phi_{V} , \derm_t   \Lie_{Z^I}  \Phi_{V} + \derm_r  \Lie_{Z^I}  \Phi_{V} >    \\
     \notag
  && +  \de^{ij}  < ( \derm_i - \frac{x_i}{r} \derm_{r}  ) \Lie_{Z^I}  \Phi_{V} , (\derm_j - \frac{x_j}{r} \derm_{r} )  \Lie_{Z^I}  \Phi_{V} > \Big)   \; .
  \eeaa
Hence, we get the announced result.
 \end{proof}

  \begin{lemma}\label{howtogetthedesirednormintheexpressionofenergyestimate}
Assume that the perturbation of the Minkowski metric is such that $ H^{\mu\nu} = g^{\mu\nu}-m^{\mu\nu}$ is bounded by a constant $C < \frac{1}{3}$\;, i.e.
\bea\label{AssumptiononHforgettingthenormintheexpressionofenergyestimate}
| H| \leq C < \frac{1}{3} \; .
\eea
Then, we have
 \beaa
 | \derm  \Lie_{Z^I}  \Phi_{V} |^2 \sim  - (m^{t t} + H^{t t} ) | \derm_t  \Lie_{Z^I}  \Phi_{V} |^2 +   (m^{ij} + H^{ij} ) < \derm_i  \Lie_{Z^I}  \Phi_{V}  , \derm_j  \Lie_{Z^I}  \Phi_{V}> \; .
  \eeaa
\end{lemma}  
  
  \begin{proof}
     See \cite{G4} for a proof.
  \end{proof}

 \begin{lemma}\label{energyestimatewithoutestimatingthetermsthatinvolveBIGHbutbydecomposingthemcorrectlysothatonecouldgettherightestimatewithtildew}
For solutions of \eqref{nonlinearsystemoftensorialwaveequations}, and for 
\bea
| H| < \frac{1}{3} \; ,
\eea
and for $\Phi$ decaying sufficiently fast at spatial infinity, we have the following energy estimate
    \beaa
   \notag
 &&     \int_{\Sigma^{ext}_{t_2} }  |\derm  \Lie_{Z^I}  \Phi_{V} |^2     \cdot   \widetilde{w} (q)  \cdot d^{3}x    + \int_{N_{t_1}^{t_2} } T_{\hat{L} t}^{(\bf{g})}   \cdot    \widetilde{w} (q) \cdot dv^{(\bf{m})}_N \\
 \notag
 &&+ \int_{t_1}^{t_2}  \int_{\Sigma^{ext}_{\tau} }  \Big(    \frac{1}{2} \Big(  | \derm_t   \Lie_{Z^I}  \Phi_{V} + \derm_r  \Lie_{Z^I}  \Phi_{V} |^2  +  \sum_{i=1}^n | ( \derm_i - \frac{x_i}{r} \derm_{r}  )\ \Lie_{Z^I}  \Phi_{V} |^2 \Big)  \cdot d\tau \cdot   \widetilde{w} ^\prime (q) d^{3}x \\
  \notag
  &\les &       \int_{\Sigma^{ext}_{t_1} }  |\derm  \Lie_{Z^I}  \Phi_{V} |^2     \cdot   \widetilde{w} (q)  \cdot d^{3}x \\
    \notag
     && +  \int_{t_1}^{t_2}  \int_{\Sigma^{ext}_{\tau} }  | H_{LL } | \cdot | \derm  \Lie_{Z^I}  \Phi_{V} |^2    \cdot d\tau \cdot   \widetilde{w} ^\prime (q) d^{3}x \\
  \notag   
     && + \int_{t_1}^{t_2}  \int_{\Sigma^{ext}_{\tau} }  |  H |  \cdot  \Big(  | \derm_t   \Lie_{Z^I}  \Phi_{V} + \derm_r  \Lie_{Z^I}  \Phi_{V} |  +   \sum_{i=1}^n  | ( \derm_i - \frac{x_i}{r} \derm_{r}  )  \Lie_{Z^I}  \Phi_{V} | \Big) \\
&& \times  | \derm  \Lie_{Z^I}  \Phi_{V} | \Big)  \cdot d\tau\cdot   \widetilde{w} ^\prime (q) d^{3}x \; .  \\
     \notag
   &&  +  \int_{t_1}^{t_2}  \int_{\Sigma^{ext}_{\tau} }  |  g^{\mu\a} \derm_{\mu } \derm_\a  \Lie_{Z^I}  \Phi_{V} | \cdot |  \derm_t  \Lie_{Z^I}  \Phi_{V} |  \cdot d\tau \cdot   \widetilde{w} (q) d^{3}x  \\
 \notag
&& + \int_{t_1}^{t_2}  \int_{\Sigma^{ext}_{\tau} } \Big(  | \derm H_{LL} |  +  | \derm_t  H + \derm_r H  |  +   \sum_{i=1}^n  | ( \derm_i - \frac{x_i}{r} \derm_{r}  ) H  | \Big) \cdot | \derm \Phi_{V} |^2 \cdot d\tau \cdot   \widetilde{w} (q) d^{3}x \\
&& + \int_{t_1}^{t_2}  \int_{\Sigma^{ext}_{\tau} }  | \derm H |  \cdot  \Big(  | \derm_t   \Lie_{Z^I}  \Phi_{V} + \derm_r  \Lie_{Z^I}  \Phi_{V} |  +   \sum_{i=1}^n  | ( \derm_i - \frac{x_i}{r} \derm_{r}  )  \Lie_{Z^I}  \Phi_{V} | \Big) \\
&& \times  | \derm  \Lie_{Z^I}  \Phi_{V} | \Big)  \cdot d\tau \cdot   \widetilde{w} (q) d^{3}x \; .  \\
 \eeaa

 \end{lemma}

 \begin{proof}
 
 Injecting the result of Lemma \ref{comutingthetermthatcarriesthederivativeoftheweight} in Corollary \ref{WeightedenergyestimateusingHandnosmallnessyet}, we obtain

     \beaa  \notag
  && \int_{N_{t_1}^{t_2} }  T_{\hat{L} t}^{(\bf{g})}  \cdot  \widetilde{w}(q) \cdot dv^{(\bf{m})}_N  +   \int_{\Sigma^{ext}_{t_2} }  \big(    - \frac{1}{2} (m^{t t} + H^{tt} ) < \derm_t  \Lie_{Z^I}  \Phi_{V} ,   \derm_t  \Lie_{Z^I}  \Phi_{V} >  \\
  && + \frac{1}{2} (m^{j i} +H^{j i} )   < \derm_j  \Lie_{Z^I}  \Phi_{V} ,   \derm_i  \Lie_{Z^I}  \Phi_{V} > \big)  \cdot   \widetilde{w} (q) \cdot d^{3}x \\
  \notag
&&    \int_{t_1}^{t_2}  \int_{\Sigma^{ext}_{\tau} }  \Big(   \frac{1}{2} \Big(  |\derm_t   \Lie_{Z^I}  \Phi_{V}+ \derm_r  \Lie_{Z^I}  \Phi_{V} |^2   \\
&&  +  \de^{ij}  < ( \derm_i - \frac{x_i}{r} \derm_{r}  )  \Lie_{Z^I}  \Phi_{V} , (\derm_j - \frac{x_j}{r} \derm_{r} )  \Lie_{Z^I}  \Phi_{V} >   \Big) \cdot   \widetilde{w} ^\prime (q) d^{3}x \\
     &=&   \int_{\Sigma^{ext}_{t_1} } \big(    - \frac{1}{2} (m^{t t} + H^{tt} )  | \derm_t  \Lie_{Z^I}  \Phi_{V} |^2   \\
     && + \frac{1}{2} (m^{j i} +H^{j i} )   < \derm_j  \Lie_{Z^I}  \Phi_{V} ,   \derm_i  \Lie_{Z^I}  \Phi_{V} > \big)   \cdot   \widetilde{w} (q)  \cdot d^{3}x  \\
     && +  \int_{t_1}^{t_2}  \int_{\Sigma^{ext}_{\tau} }  \Big(  2 H^{\underline{L}\a}< \derm_\a  \Lie_{Z^I}  \Phi_{V} ,   \derm_t  \Lie_{Z^I}  \Phi_{V} > \\
     &&  - \frac{1}{2}   \cdot H^{\a\b}  < \derm_\a \Phi_{V} ,   \derm_\b  \Lie_{Z^I}  \Phi_{V} > \Big) \cdot d\tau \cdot   \widetilde{w}^\prime (q) d^{3}x 
\eeaa
\beaa
   &&  -  \int_{t_1}^{t_2}  \int_{\Sigma^{ext}_{\tau} } \Big(  < g^{\mu\a} \derm_{\mu } \derm_\a  \Lie_{Z^I}  \Phi_{V},   \derm_t \Lie_{Z^I}  \Phi_{V}>  \\&& +( \derm_{\mu } H^{\mu\a} ) \cdot < \derm_\a  \Lie_{Z^I}  \Phi_{V},   \derm_t  \Lie_{Z^I}  \Phi_{V}> \\
 \notag
&&- \frac{1}{2} m^{\mu}_{\;\;\; t}  \cdot  ( \derm_{\mu } H^{\a\b} ) \cdot < \derm_\a  \Lie_{Z^I}  \Phi_{V},   \derm_\b  \Lie_{Z^I}  \Phi_{V} >   \Big) \cdot d\tau \cdot   \widetilde{w} (q) d^{3}x \; .  \\
 \eeaa
 
 Using Lemma \ref{howtogetthedesirednormintheexpressionofenergyestimate}, we get for $ | H| \leq C < \frac{1}{3} $\,, the following energy estimate

   \beaa
   \notag
 &&     \int_{\Sigma^{ext}_{t_2} }  |\derm  \Lie_{Z^I}  \Phi_{V} |^2     \cdot w(q)  \cdot d^{3}x    + \int_{N_{t_1}^{t_2} }  T_{\hat{L} t}^{(\bf{g})}   \cdot    \widetilde{w} (q) \cdot dv^{(\bf{m})}_N \\
 \notag
 &&+ \int_{t_1}^{t_2}  \int_{\Sigma^{ext}_{\tau} }  \Big(    \frac{1}{2} \Big(  | \derm_t   \Lie_{Z^I}  \Phi_{V} + \derm_r  \Lie_{Z^I}  \Phi_{V} |^2  +   \sum_{i=1}^3  | ( \derm_i - \frac{x_i}{r} \derm_{r}  )  \Lie_{Z^I}  \Phi_{V} |^2 \Big)  \cdot d\tau \cdot   \widetilde{w}^\prime (q) d^{3}x \\
  \notag
  &\leq &       \int_{\Sigma^{ext}_{t_1} }  |\derm  \Lie_{Z^I}  \Phi_{V} |^2     \cdot \widetilde{w}(q)  \cdot d^{3}x \\
    \notag
     && +  \int_{t_1}^{t_2}  \int_{\Sigma^{ext}_{\tau} }  | 2 H^{\underline{L}\a}< \derm_\a  \Lie_{Z^I}  \Phi_{V} ,   \derm_t  \Lie_{Z^I}  \Phi_{V} > \\
      \notag
     &&  - \frac{1}{2}   \cdot H^{\a\b}  < \derm_\a  \Lie_{Z^I}  \Phi_{V} ,   \derm_\b  \Lie_{Z^I}  \Phi_{V}>  | \cdot d\tau \cdot \widetilde{w}^\prime (q) d^{3}x \\
     \notag
   &&  +  \int_{t_1}^{t_2}  \int_{\Sigma^{ext}_{\tau} }  | < g^{\mu\a} \derm_{\mu } \derm_\a  \Lie_{Z^I}  \Phi_{V},   \derm_t  \Lie_{Z^I}  \Phi_{V} >  \\
    \notag
   && +( \derm_{\mu } H^{\mu\a} ) \cdot < \derm_\a  \Lie_{Z^I}  \Phi_{V} ,   \derm_t  \Lie_{Z^I}  \Phi_{V}> \\
&&- \frac{1}{2} m^{\mu}_{\;\;\; t}  \cdot  ( \derm_{\mu } H^{\a\b} ) \cdot < \derm_\a  \Lie_{Z^I}  \Phi_{V},   \derm_\b  \Lie_{Z^I}  \Phi_{V} >   |  \cdot d\tau \cdot \widetilde{w}(q) d^{3}x \; .  
 \eeaa

 Since $ m^{\mu}_{\;\;\; t} = m^{\mu\a} \cdot m_{\a t}  =m^{\mu t} \cdot m_{t t} = - m^{\mu t} $, we estimate
\bea\label{estimateonthedivergenceoftheenergymomentumtensordecomposedinnullframe}
 \notag
 && | < g^{\mu\a} \derm_{\mu } \derm_\a  \Lie_{Z^I}  \Phi_{V} ,   \derm_t  \Lie_{Z^I}  \Phi_{V} >  \\
  \notag
 && +( \derm_{\mu } H^{\mu\a} ) \cdot < \derm_\a  \Lie_{Z^I}  \Phi_{V},   \derm_t  \Lie_{Z^I}  \Phi_{V}> \\
 \notag
&&- \frac{1}{2} m^{\mu}_{\;\;\; t}  \cdot  ( \derm_{\mu } H^{\a\b} ) \cdot < \derm_\a  \Lie_{Z^I}  \Phi_{V},   \derm_\b  \Lie_{Z^I}  \Phi_{V} > | \\
 \notag
 &\les& |  g^{\mu\a} \derm_{\mu } \derm_\a  \Lie_{Z^I}  \Phi_{V} | \cdot |  \derm_t  \Lie_{Z^I}  \Phi_{V}  |  \\
  \notag
 && +| ( \derm_{\mu } H^{\mu\a} ) \cdot < \derm_\a  \Lie_{Z^I}  \Phi_{V},   \derm_t  \Lie_{Z^I}  \Phi_{V} > | \\
 &&+ | \frac{1}{2} m^{t t} \cdot  ( \derm_{t} H^{\a\b} ) \cdot < \derm_\a  \Lie_{Z^I}  \Phi_{V} ,   \derm_\b  \Lie_{Z^I}  \Phi_{V}>  |  \, .
\eea

We have by decomposing simply in a null frame, or by equivalently by using Lemma 4.2 from \cite{LR10}, that 
 \beaa
\notag
&& | ( \derm_{\mu } H^{\mu\a} ) \cdot < \derm_\a  \Lie_{Z^I}  \Phi_{V},   \derm_t  \Lie_{Z^I}  \Phi_{V}> | \\
 \notag
&&+ | \frac{1}{2} m^{t t} \cdot  ( \derm_{t} H^{\a\b} ) \cdot < \derm_\a  \Lie_{Z^I}  \Phi_{V},   \derm_\b  \Lie_{Z^I}  \Phi_{V} >  | \\
 \notag
&\les&  | ( \derm_{\mu } H^{\mu\a} ) \cdot < \derm_\a  \Lie_{Z^I}  \Phi_{V} ,   \derm_t  \Lie_{Z^I}  \Phi_{V} > | \\
 \notag
&&+ |  ( \derm_{t} H^{\a\b} ) \cdot < \derm_\a  \Lie_{Z^I}  \Phi_{V},   \derm_\b  \Lie_{Z^I}  \Phi_{V} >  | 
\eeaa
\beaa
&\les& \Big( | \derm H_{LL} |  + | \derm_t   H + \derm_r  H |  +   \sum_{i=1}^n  | ( \derm_i - \frac{x_i}{r} \derm_{r}  )  H | \Big) \cdot | \derm  \Lie_{Z^I}  \Phi_{V} |^2 \\
\notag
&& +  | \derm H | \cdot \Big(  | \derm_t   \Lie_{Z^I}  \Phi_{V} + \derm_r  \Lie_{Z^I}  \Phi_{V} |  +   \sum_{i=1}^n  | ( \derm_i - \frac{x_i}{r} \derm_{r}  )  \Lie_{Z^I}  \Phi_{V} | \Big)  \cdot  | \derm  \Lie_{Z^I}  \Phi_{V} | \; . 
\eeaa

Similarly, by decomposing in null frame, as in Lemma 4.2 in \cite{LR10}, where $\underline{L}$ is the vector conjugate of $L$, we get
   \beaa
\notag
&& | 2 H^{\underline{L}\a}< \derm_\a  \Lie_{Z^I}  \Phi_{V} ,   \derm_t  \Lie_{Z^I}  \Phi_{V}>  - \frac{1}{2}   \cdot H^{\a\b}  < \derm_\a  \Lie_{Z^I}  \Phi_{V},   \derm_\b  \Lie_{Z^I}  \Phi_{V} >  | \\
\notag
  &\leq& | H^{\underline{L}\a}< \derm_\a  \Lie_{Z^I}  \Phi_{V},   \derm_t  \Lie_{Z^I}  \Phi_{V} > |  +| H^{\a\b}  < \derm_\a  \Lie_{Z^I}  \Phi_{V} ,   \derm_\b  \Lie_{Z^I}  \Phi_{V} > |  \\
  \notag
 &\les& | H_{LL } | \cdot | \derm  \Lie_{Z^I}  \Phi_{V} |^2 \\
 \notag
 && + |  H | \cdot \Big(  | \derm_t   \Lie_{Z^I}  \Phi_{V} + \derm_r  \Lie_{Z^I}  \Phi_{V} |  +   \sum_{i=1}^n  | ( \derm_i - \frac{x_i}{r} \derm_{r}  )  \Lie_{Z^I}  \Phi_{V} | \Big) \cdot | \derm  \Lie_{Z^I}  \Phi_{V}  | \; . 
 \eeaa
Inserting, we obtain the desired result.
\end{proof}

Using Lemmas \ref{equaivalenceoftildewandtildeandofderivativeoftildwandderivativeofhatw} and \ref{derivativeoftildwandrelationtotildew} and inserting in Lemma \ref{energyestimatewithoutestimatingthetermsthatinvolveBIGHbutbydecomposingthemcorrectlysothatonecouldgettherightestimatewithtildew}, we obtain the proof of \eqref{TheenerhyestimatewithtermsinvolvingHandderivativeofHandwithwandhatw}.

 \section{The commutator term for the non-linear wave equations}\

Now, in Lemma \ref{energyestimatewithoutestimatingthetermsthatinvolveBIGHbutbydecomposingthemcorrectlysothatonecouldgettherightestimatewithtildew}, we would like to prove \eqref{commutationformaulamoreprecisetoconservegpodcomponentsstructure}. The goal is estimate the term $ |  g^{\mu\a} \derm_{\mu } \derm_\a  \Lie_{Z^I}  \Phi_{V} | $ by estimating the commutator term $$ | g^{\la\mu}    \derm_{\la}   \derm_{\mu} \Lie_{Z^I} \Phi_{V}    - \Lie_{Z^I}  ( g^{\la\mu} \derm_{\la}   \derm_{\mu}     \Phi_{V} )  | \, .$$
 We establish a new decoupled commutation formula that is more refined than the one derived by Lindblad-Rodnianski in their Proposition 5.3 in \cite{LR10}, in order to be able to have a commutation formula for the tangential components alone.

\begin{lemma}\label{theexactcommutatortermwithdependanceoncomponents}
We have for all $I$\;, or any $U \in {\cal U}$\;,
\beaa
\notag
&& \Lie_{Z^I}  ( g^{\la\mu} \derm_{\la}   \derm_{\mu}     \Phi_{U} ) - g^{\la\mu}    \derm_{\la}   \derm_{\mu}  (  \Lie_{Z^I} \Phi_{U}  )  \\
\notag
&=&  \sum_{I_1 + I_2 = I, \; I_2 \neq I} \hat{c}(I_1) \cdot m^{\la\mu} \cdot \derm_{\la}   \derm_{\mu} (  \Lie_{Z^{I_2}}  \Phi_{U} ) \\
\notag
&& +  \sum_{I_2 + I_4 + I_5 + I_6  = I, \; I_2 \neq I }   \hat{c}(I_6) \cdot   \hat{c}(I_5)   \cdot \Big(   \frac{1}{4}   ( \Lie_{Z^{I_4}}   H)_{L  L}    \cdot  \derm_{\underline{L}}   \derm_{ \underline{L}}   (  \Lie_{Z^{I_2}}  \Phi_{U} ) \\
\notag
&&  + \frac{1}{4}  ( \Lie_{Z^{I_4}}   H)_{L  \underline{L}}    \cdot  \derm_{\underline{L}}   \derm_{L}   (  \Lie_{Z^{I_2}}  \Phi_{U} )    - \frac{1}{2}   ( \Lie_{Z^{I_4}}   H)_{L e_A }    \cdot  \derm_{\underline{L}}   \derm_{e_A}    (  \Lie_{Z^{I_2}}  \Phi_{U} )    \\
   \notag
 && - \frac{1}{2} m^{\mu \b}   ( \Lie_{Z^{I_4}}   H)_{\underline{L}\b}  \cdot    \derm_{L}   \derm_{\mu}     (  \Lie_{Z^{I_2}}  \Phi_{U} )   +    m^{\mu \b}   ( \Lie_{Z^{I_4}}   H)_{e_A \b}   \cdot   \derm_{e_A}   \derm_{\mu}    (  \Lie_{Z^{I_2}}  \Phi_{U} )   \Big)  \; .\\
 \eeaa
 
 \end{lemma}
 
 \begin{proof}
 
We compute for any $U \in {\cal U}$\,, 
\beaa
&& \Lie_{Z^{I}}  ( g^{\la\mu} \derm_{\la}   \derm_{\mu}     \Phi_{U} )  \\
&=&  \Lie_{Z^{I}}  \Big( (m^{\la\mu} + H^{\la\mu} ) \cdot \derm_{\la}   \derm_{\mu}     \Phi_{U} \Big) \\
&=& \sum_{I_1 + I_2 = I} \Big( \Lie_{Z^{I_1}} m^{\la\mu} + \Lie_{Z^{I_1}}  H^{\la\mu} \Big)  \cdot  \Lie_{Z^{I_2}}  \derm_{\la}   \derm_{\mu}     \Phi_{U} \\
&=&\sum_{I_1 + I_2 = I} \hat{c}(I_1) \cdot m^{\la\mu} \cdot \Lie_{Z^{I_2}}  \derm_{\la}   \derm_{\mu}  \Phi_{U}  + \sum_{I_1 + I_2 = I} ( \Lie_{Z^{I_1}}  H^{\la\mu} ) \cdot   \Lie_{Z^{I_2}}  \derm_{\la}   \derm_{\mu}  \Phi_{U} \; .
\eeaa

We note that 
\beaa
&& \Lie_{Z^{I_1}}  H^{\la\mu} \\
 &=&  \Lie_{Z^{I_1}}(  m^{\la\a} \cdot m^{\mu\b} \cdot  H_{\a\b}  ) \\
  &=&\sum_{I_3 + I_4   = I_1 }  \Lie_{Z^{I_3}} (  m^{\la\a}  \cdot m^{\mu\b} ) \cdot  \Lie_{Z^{I_4}} (  H_{\a\b}  ) = \sum_{I_3 + I_4   = I_1 } \Big( \sum_{I_5 + I_6   = I_3 } (  \Lie_{Z^{I_5}}  m^{\la\a}  \cdot  \Lie_{Z^{I_6}}  m^{\mu\b} ) \cdot  \Lie_{Z^{I_4}} (  H_{\a\b}  ) \Big)    \\
  &=&  \sum_{ I_4 + I_5 + I_6  = I_1 } (   \hat{c}(I_5) \cdot  m^{\la\a} )  \cdot (  \hat{c}(I_6) \cdot m^{\mu\b} )  \cdot (   \Lie_{Z^{I_4}}    H_{\a\b}  )  \; .
\eeaa
With the notation $ m^{\la\a}  \cdot m^{\mu\b}  \cdot ( \Lie_{Z^I}   H_{\a\b}  )  =  ( \Lie_{Z^I}   H)^{\la\mu}  $\,, and using the fact that the Lie derivatives in the direction of Minkowski vector fields commute with $\derm$\,, we can then write
\beaa
&& \Lie_{Z^I}  ( g^{\la\mu} \derm_{\la}   \derm_{\mu}     \Phi_{U} )  \\
&=& \sum_{I_1 + I_2 = I} \hat{c}(I_1) \cdot m^{\la\mu} \cdot \derm_{\la}   \derm_{\mu} (  \Lie_{Z^{I_2}}  \Phi_{U} ) \\
&& +  \sum_{I_2 + I_4 + I_5 + I_6  = I } (   \hat{c}(I_5) \cdot  m^{\la\a} )  \cdot (  \hat{c}(I_6) \cdot m^{\mu\b} )  \cdot (   \Lie_{Z^{I_4}}    H_{\a\b}  ) \cdot  \derm_{\la}   \derm_{\mu} (  \Lie_{Z^{I_2}}  \Phi_{U} )   \\
&=& \sum_{I_1 + I_2 = I} \hat{c}(I_1) \cdot m^{\la\mu} \cdot \derm_{\la}   \derm_{\mu} (  \Lie_{Z^{I_2}}  \Phi_{U} ) \\
&& +  \sum_{I_2 + I_4 + I_5 + I_6  = I }   \hat{c}(I_5) \cdot   \hat{c}(I_6)   \cdot (   \Lie_{Z^{I_4}}     H)^{\la\mu} \cdot  \derm_{\la}   \derm_{\mu} (  \Lie_{Z^{I_2}}  \Phi_{U} )    \\
&=&    m^{\la\mu} \cdot \derm_{\la}   \derm_{\mu} (  \Lie_{Z^{I}}  \Phi_{U} ) +       H^{\la\mu} \cdot  \derm_{\la}   \derm_{\mu} (  \Lie_{Z^{I}}  \Phi_{U} )    \\
&& + \sum_{I_1 + I_2 = I, \; I_2 \neq I} \hat{c}(I_1) \cdot m^{\la\mu} \cdot \derm_{\la}   \derm_{\mu} (  \Lie_{Z^{I_2}}  \Phi_{U} ) \\
&& +  \sum_{I_2 + I_4 + I_5 + I_6  = I, \; I_2 \neq I }   \hat{c}(I_5) \cdot   \hat{c}(I_6)   \cdot (   \Lie_{Z^{I_4}}     H)^{\la\mu} \cdot  \derm_{\la}   \derm_{\mu} (  \Lie_{Z^{I_2}}  \Phi_{U} )    \; .
\eeaa

Therefore, 
\beaa
\notag
&& \Lie_{Z^I}  ( g^{\la\mu} \derm_{\la}   \derm_{\mu}     \Phi_{U} ) - g^{\la\mu}    \derm_{\la}   \derm_{\mu}  (  \Lie_{Z^I} \Phi_{U}  )  \\
\notag
&=&  \sum_{I_1 + I_2 = I, \; I_2 \neq I} \hat{c}(I_1) \cdot m^{\la\mu} \cdot \derm_{\la}   \derm_{\mu} (  \Lie_{Z^{I_2}}  \Phi_{U} ) \\
\notag
&& +  \sum_{I_2 + I_4 + I_5 + I_6  = I, \; I_2 \neq I }   \hat{c}(I_5) \cdot   \hat{c}(I_6)   \cdot (   \Lie_{Z^{I_4}}     H)^{\la\mu} \cdot  \derm_{\la}   \derm_{\mu} (  \Lie_{Z^{I_2}}  \Phi_{U} )   \; . 
\eeaa

Now, to estimate the term with  $H$\,, we decompose the contractions in the null frame $\cal U$\,, 
\beaa
 && ( \Lie_{Z^{I_4}}   H)^{\la\mu}  \cdot  \derm_{\la}   \derm_{\mu}    (  \Lie_{Z^{I_2}}  \Phi_{U} )  \\
 & =&  ( \Lie_{Z^{I_4}}   H)^{\underline{L}\mu}  \cdot   \derm_{\underline{L}}   \derm_{\mu}   (  \Lie_{Z^{I_2}}  \Phi_{U} )  +  ( \Lie_{Z^{I_4}}   H)^{L \mu}   \cdot  \derm_{L}   \derm_{\mu}     (  \Lie_{Z^{I_2}}  \Phi_{U} )  \\
 && +  ( \Lie_{Z^{I_4}}   H)^{e_A \mu}   \cdot  \derm_{e_A}   \derm_{\mu}   (  \Lie_{Z^{I_2}}  \Phi_{U} )  \\
 & =&  m^{\underline{L} L} m^{\mu \b}  ( \Lie_{Z^{I_4}}   H)_{L\b}    \cdot  \derm_{\underline{L}}   \derm_{\mu}   (  \Lie_{Z^{I_2}}  \Phi_{U} )  \\
 && +m^{L \underline{L}} m^{\mu \b}   ( \Lie_{Z^{I_4}}   H)_{\underline{L}\b}   \cdot  \derm_{L}   \derm_{\mu} (  \Lie_{Z^{I_2}}  \Phi_{U} )   +   m^{e_A e_A} m^{\mu \b}   ( \Lie_{Z^{I_4}}   H)_{e_A \b}     \cdot \derm_{e_A}   \derm_{\mu}   (  \Lie_{Z^{I_2}}  \Phi_{U} )  \; .
 \eeaa
 
 Using the fact that  $m^{L \underline{L}} = - \frac{1}{2} $ and $m^{e_A e_B}  =  \delta_{AB} $\,, we obtain
\beaa
 && ( \Lie_{Z^{I_4}}   H)^{\la\mu}  \cdot  \derm_{\la}   \derm_{\mu}   (  \Lie_{Z^{I_2}}  \Phi_{U} )  \\
 & =&   - \frac{1}{2} m^{\mu \b}  ( \Lie_{Z^{I_4}}   H)_{L\b}    \cdot  \derm_{\underline{L}}   \derm_{\mu}   (  \Lie_{Z^{I_2}}  \Phi_{U} )   \\
 && - \frac{1}{2} m^{\mu \b}   ( \Lie_{Z^{I_4}}   H)_{\underline{L}\b}  \cdot    \derm_{L}   \derm_{\mu}     (  \Lie_{Z^{I_2}}  \Phi_{U} )  +    m^{\mu \b}   ( \Lie_{Z^{I_4}}   H)_{e_A \b}   \cdot   \derm_{e_A}   \derm_{\mu}   (  \Lie_{Z^{I_2}}  \Phi_{U} )  \\
  & =&   - \frac{1}{2} m^{ \underline{L} L }  ( \Lie_{Z^{I_4}}   H)_{L  L}    \cdot  \derm_{\underline{L}}   \derm_{ \underline{L}}   (  \Lie_{Z^{I_2}}  \Phi_{U} )   - \frac{1}{2} m^{L  \underline{L}}  ( \Lie_{Z^{I_4}}   H)_{L  \underline{L}}    \cdot  \derm_{\underline{L}}   \derm_{L}   (  \Lie_{Z^{I_2}}  \Phi_{U} )   \\
  && - \frac{1}{2} m^{e_A e_A}  ( \Lie_{Z^{I_4}}   H)_{L e_A }    \cdot  \derm_{\underline{L}}   \derm_{e_A} (  \Lie_{Z^{I_2}}  \Phi_{U} )    \\
 && - \frac{1}{2} m^{\mu \b}   ( \Lie_{Z^{I_4}}   H)_{\underline{L}\b}  \cdot    \derm_{L}   \derm_{\mu}   (  \Lie_{Z^{I_2}}  \Phi_{U} )  +    m^{\mu \b}   ( \Lie_{Z^{I_4}}   H)_{e_A \b}   \cdot   \derm_{e_A}   \derm_{\mu}    (  \Lie_{Z^{I_2}}  \Phi_{U} )  \; .
 \eeaa
 Thus,
 \beaa
 \notag
 && ( \Lie_{Z^{I_4}}   H)^{\la\mu}  \cdot  \derm_{\la}   \derm_{\mu}    (  \Lie_{Z^{I_2}}  \Phi_{U} )   \\
  \notag
  & =&   \frac{1}{4}   ( \Lie_{Z^{I_4}}   H)_{L  L}    \cdot  \derm_{\underline{L}}   \derm_{ \underline{L}}   (  \Lie_{Z^{I_2}}  \Phi_{U} )   + \frac{1}{4}  ( \Lie_{Z^{I_4}}   H)_{L  \underline{L}}    \cdot  \derm_{\underline{L}}   \derm_{L}   (  \Lie_{Z^{I_2}}  \Phi_{U} )   \\
   \notag
  && - \frac{1}{2}   ( \Lie_{Z^{I_4}}   H)_{L e_A }    \cdot  \derm_{\underline{L}}   \derm_{e_A}    (  \Lie_{Z^{I_2}}  \Phi_{U} )    \\
   \notag
 && - \frac{1}{2} m^{\mu \b}   ( \Lie_{Z^{I_4}}   H)_{\underline{L}\b}  \cdot    \derm_{L}   \derm_{\mu}     (  \Lie_{Z^{I_2}}  \Phi_{U} )   +    m^{\mu \b}   ( \Lie_{Z^{I_4}}   H)_{e_A \b}   \cdot   \derm_{e_A}   \derm_{\mu}    (  \Lie_{Z^{I_2}}  \Phi_{U} )   \; . 
 \eeaa
  We get then the desired result.
  
 \end{proof}
 
\begin{lemma}\label{decayrateforfullderivativeintermofZ}
We have the following well-known inequality for all $t \geq 0$ and for all $q \in \R$,
\beaa
(1 +  |q| ) \cdot |\derm \Lie_{Z^I}  \Phi | \les \sum_{|J| \leq |I| + 1} | \Lie_{Z^J}  \Phi |  \,.
\eeaa
\end{lemma}

\begin{definition}\label{defrestrictedderivativesintermsofZ}
We define for $i \in \{1, 2, 3\} $\;,
\bea
\notag
\rpa_{i} &:=& \pa_i - \frac{x_i}{r} \pa_{r} \; .
\eea
\end{definition}

\begin{lemma}\label{restrictedderivativesintermsofZ}

Let $Z_{\a\b}$ be the Minkowski vector fields as in Definition \label{DefinitionofMinkowskivectorfields}. For restricted partial derivatives defines as in Definition \ref{defrestrictedderivativesintermsofZ}, we have for $i \in \{1, 2, 3\} $\;,
\bea
\rpa_{i} &=& \frac{  - \frac{x_i}{r} \frac{x^j}{r}   Z_{0j}  +  Z_{0i}    }{  t  }  \;, 
\eea
and we also have
 \bea
\rpa_i  &=&   \frac{x^j}{r^2}   Z_{ij} \; .
 \eea

\end{lemma}

\begin{proof}
See for example \cite{G6}.
\end{proof}

 \begin{lemma}\label{commutationformaulamoreprecisetoconservegpodcomponentsstructure}
 Let  $\Phi_{\mu}$ be a one-tensor valued in the Lie algebra or a scalar. Then, we have for all $I$\,, and for any $V \in \cal T$, 
 \bea\label{morerefinedcommutationformularusingonlytangentialcomponennts}
\notag
&&| \Lie_{Z^I}  ( g^{\la\mu} \derm_{\la}   \derm_{\mu}     \Phi_{V} ) - g^{\la\mu}    \derm_{\la}   \derm_{\mu}  (  \Lie_{Z^I} \Phi_{V}  ) |  \\
\notag
&\les&  \sum_{|K| < |I| }  | g^{\la\mu} \cdot \derm_{\la}   \derm_{\mu} (  \Lie_{Z^{K}}  \Phi_{V} ) | \\
\notag
&& +  \sum_{|J| + |K| \leq |I|, \; |K| < |I| }  \Big(   |   ( \Lie_{Z^{J}}   H)_{L  L} |   \cdot    \frac{1}{(1+t+|q|)} \cdot  \sum_{|M| \leq |K|+1}  | \derm ( \Lie_{Z^M}  \Phi )  |  \\
   \notag
 && +    |   ( \Lie_{Z^{J}}   H)_{L  L}   | \cdot   \frac{1}{(1+|q|)} \cdot   \sum_{|M| \leq |K|+1}    \sum_{ V^\prime \in \cal T } | \derm   ( \Lie_{Z^M}  \Phi _{V^\prime} ) |   \\
\notag
 && +   | ( \Lie_{Z^{J}}   H)_{L  \underline{L}} |   \cdot | \derm_{L}    \derm_{\underline{L}}  (  \Lie_{Z^{K}}  \Phi_{V} )  | \\
\notag
    &&  +    \frac{1}{(1 + t + |q|)  }  \cdot   |  ( \Lie_{Z^{J}}   H)_{L e_A }   | \cdot | \sum_{|M| \leq |K|+1} | \derm (  \Lie_{Z^{M}}  \Phi )| \\
    \notag
 && +       | m^{\mu \b}   ( \Lie_{Z^{J}}   H)_{\underline{L}\b}  \cdot    \derm_{L}   \derm_{\mu}     (  \Lie_{Z^{K}}  \Phi_{V} )  | + |   m^{\mu \b}   ( \Lie_{Z^{J}}   H)_{e_A \b}   \cdot   \derm_{e_A}   \derm_{\mu}    (  \Lie_{Z^{K}}  \Phi_{V} )   | \Big) \; .\\
  \eea
 \end{lemma}
 
  \begin{proof}
 From Lemma \ref{theexactcommutatortermwithdependanceoncomponents}, we immediately get that for any $U \in {\cal U}$\,,
 \bea
\notag
&&| \Lie_{Z^I}  ( g^{\la\mu} \derm_{\la}   \derm_{\mu}     \Phi_{U} ) - g^{\la\mu}    \derm_{\la}   \derm_{\mu}  (  \Lie_{Z^I} \Phi_{U}  ) |  \\
\notag
&\les&  \sum_{ |K| < |I| }  | m^{\la\mu} \cdot \derm_{\la}   \derm_{\mu} (  \Lie_{Z^{K}}  \Phi_{U} ) | \\
\notag
&& +  \sum_{J, \, K, \,  |J| + |K| \leq |I|, \; |K| < |I| }  \Big( |   ( \Lie_{Z^{J}}   H)_{L  L}    \cdot  \derm_{\underline{L}}   \derm_{ \underline{L}}   (  \Lie_{Z^{K}}  \Phi_{U} ) | \\
\notag
&&  + | ( \Lie_{Z^{J}}   H)_{L  \underline{L}}    \cdot  \derm_{\underline{L}}   \derm_{L}   (  \Lie_{Z^{K}}  \Phi_{U} )  |  + |  ( \Lie_{Z^{J}}   H)_{L e_A }    \cdot  \derm_{\underline{L}}   \derm_{e_A}    (  \Lie_{Z^{K}}  \Phi_{U} )  |   \\
   \notag
 && + | m^{\mu \b}   ( \Lie_{Z^{J}}   H)_{\underline{L}\b}  \cdot    \derm_{L}   \derm_{\mu}     (  \Lie_{Z^{K}}  \Phi_{U} )  | + |   m^{\mu \b}   ( \Lie_{Z^{J}}   H)_{e_A \b}   \cdot   \derm_{e_A}   \derm_{\mu}    (  \Lie_{Z^{K}}  \Phi_{U} )   | \Big)  \; ,
 \eea
and also, from the proof of Lemma \ref{theexactcommutatortermwithdependanceoncomponents}, we get, in addition, that for any $U \in {\cal U}$\;,
   \bea
\notag
&&| \Lie_{Z^I}  ( g^{\la\mu} \derm_{\la}   \derm_{\mu}     \Phi_{U} ) - g^{\la\mu}    \derm_{\la}   \derm_{\mu}  (  \Lie_{Z^I} \Phi_{U}  ) |  \\
\notag
&\les&  \sum_{ |K| < |I| }  | g^{\la\mu} \cdot \derm_{\la}   \derm_{\mu} (  \Lie_{Z^{K}}  \Phi_{U} ) | \\
\notag
&& +  \sum_{J, \, K, \,  |J| + |K| \leq |I|, \; |K| < |I| }  \Big( |   ( \Lie_{Z^{J}}   H)_{L  L}    \cdot  \derm_{\underline{L}}   \derm_{ \underline{L}}   (  \Lie_{Z^{K}}  \Phi_{U} ) | \\
\notag
&&  + | ( \Lie_{Z^{J}}   H)_{L  \underline{L}}    \cdot  \derm_{\underline{L}}   \derm_{L}   (  \Lie_{Z^{K}}  \Phi_{U} )  |  + |  ( \Lie_{Z^{J}}   H)_{L e_A }    \cdot  \derm_{\underline{L}}   \derm_{e_A}    (  \Lie_{Z^{K}}  \Phi_{U} )  |   \\
   \notag
 && + | m^{\mu \b}   ( \Lie_{Z^{J}}   H)_{\underline{L}\b}  \cdot    \derm_{L}   \derm_{\mu}     (  \Lie_{Z^{K}}  \Phi_{U} )  | + |   m^{\mu \b}   ( \Lie_{Z^{J}}   H)_{e_A \b}   \cdot   \derm_{e_A}   \derm_{\mu}    (  \Lie_{Z^{K}}  \Phi_{U} )   | \Big)  \; .
 \eea
Now, we are going to estimate the terms one by one.\\
  
 \textbf{The term $   ( \Lie_{Z^{I_4}}   H)_{L  L}    \cdot  \derm_{\underline{L}}   \derm_{ \underline{L}}   (  \Lie_{Z^{I_2}}  \Phi_{U} ) $}\,:\\
 
 We have, see Lemma 3.3 in \cite{BFJST}, that for any tensor $\Psi_{UV}$ and for any $U \in {\cal U}$ and for any $V \in {\cal T}$,
 \bea\label{estimateforgradientoftensorestaimetedbyLiederivativesoftensorwithrefinedcomponentsforthebadfactor}
 \notag
 |\derm \Psi_{UV} | \les \sum_{|I| \leq 1}  \frac{1}{(1+t+|q|)} \cdot | \Lie_{Z^I} \Psi |  +  \sum_{U^\prime \in {\cal U},  V^\prime \in {\cal T} } \sum_{|I| \leq 1}  \frac{1}{(1+|q|)} \cdot  | \Lie_{Z^I} \Psi_{U^\prime V^\prime} | \; . \\
 \eea
 We can then apply the above for $\Psi_{UV} = \derm_{U} (  \Lie_{Z^{K}}   \Phi_{V} ) $\,, and we get
  \beaa
 && |\derm \derm_{U} (  \Lie_{Z^{K}}   \Phi_{V} ) | \\
 &\les& \sum_{|M| \leq 1}  \frac{1}{(1+t+|q|)} \cdot | \Lie_{Z^M} \derm (  \Lie_{Z^{K}}   \Phi)  |  +    \sum_{U^\prime \in {\cal U},  V^\prime \in {\cal T} } \sum_{|M| \leq 1}  \frac{1}{(1+|q|)} \cdot  | \Lie_{Z^M} \derm_{U^\prime } (  \Lie_{Z^{K}}   \Phi_{V^\prime } ) | \; .
 \eeaa
 Using the commutation of the Lie derivative in the direction of Minkowski vector fields with the covariant derivative of Minkowski $\derm$, we obtain 
  \bea
  \notag
 && | \derm_{\underline{L}}   \derm_{ \underline{L}}   (  \Lie_{Z^{K}}  \Phi_{V} )   | \\
   \notag
 &\les &\sum_{|M| \leq |K|+1}  \frac{1}{(1+t+|q|)} \cdot | \derm ( \Lie_{Z^M}  \Phi )  |  +  \sum_{ V^\prime \in {\cal T} }   \sum_{|M| \leq |K|+1}  \frac{1}{(1+|q|)} \cdot  | \derm   ( \Lie_{Z^M}  \Phi _{V^\prime} ) | \; .
 \eea
 Therefore, for all $V \in {\cal T}$, 
 \bea
 \notag
 && \sum_{  |J| + |K| \leq |I|, \; |K| < |I| }  |   ( \Lie_{Z^{J}}   H)_{L  L}    \cdot  \derm_{\underline{L}}   \derm_{ \underline{L}}   (  \Lie_{Z^{K}}  \Phi_{V} ) | \\
  \notag
  &\les&  \sum_{  |J| + |K| \leq |I|, \; |K| < |I| }    |   ( \Lie_{Z^{J}}   H)_{L  L} |   \cdot    \sum_{|M| \leq |K|+1}  \frac{1}{(1+t+|q|)} \cdot | \derm ( \Lie_{Z^M}  \Phi )  |  \\
   \notag
 && + \sum_{  |J| + |K| \leq |I|, \; |K| < |I| }   |   ( \Lie_{Z^{J}}   H)_{L  L}   | \cdot   \sum_{|M| \leq |K|+1}  \frac{1}{(1+|q|)} \cdot   \big( \sum_{ V^\prime \in {\cal T} } | \derm   ( \Lie_{Z^M}  \Phi _{V^\prime} ) | \big)  \;. 
 \eea
 Since, as we have already shown in Lemma \ref{decayrateforfullderivativeintermofZ}, 
  \beaa
 |\derm \Psi_{UV} | \les \sum_{|I| \leq 1}  \frac{1}{(1+|q|)} \cdot  | \Lie_{Z^I} \Psi | \; ,
 \eeaa
 we also get that for all $U \in {\cal U}$\,,
  \bea
   \notag
 && \sum_{  |J| + |K| \leq |I|, \; |K| < |I| }  |   ( \Lie_{Z^{J}}   H)_{L  L}    \cdot  \derm_{\underline{L}}   \derm_{ \underline{L}}   (  \Lie_{Z^{K}}  \Phi_{U} ) | \\
  \notag
  &\les&  \sum_{  |J| + |K| \leq |I|, \; |K| < |I| }   |   ( \Lie_{Z^{J}}   H)_{L  L}   | \cdot   \sum_{|M| \leq |K|+1}  \frac{1}{(1+|q|)} \cdot    | \derm   ( \Lie_{Z^M}  \Phi  ) |  \;. \\
 \eea

 \textbf{The term $ \derm_{\underline{L}}   \derm_{L}   (  \Lie_{Z^{K}}  \Phi_{U} )$}\,:\\
 
We have
\beaa
 && \derm_{\underline{L}}   \derm_{L}   (  \Lie_{Z^{K}}  \Phi_{U} ) \\
  &=&   \pa_{\underline{L}}  \Big(  \derm_{L}   (  \Lie_{Z^{K}}  \Phi_{U} ) \Big) -    \derm_{ \derm_{\underline{L}} L }   (  \Lie_{Z^{K}}  \Phi_{U} ) -    \derm_{L }   (  \Lie_{Z^{K}}  \Phi_{ \derm_{\underline{L}}  U} )  \; .
  \eeaa
We have shown in the Appendix of \cite{G2} (by taking in \cite{G2}, a zero Schwarzschild mass, which gives a Minkowski metric), that for all $U \in {{\cal U}}$, we have $ \derm_{\underline{L}} U  =  \derm_{L } U = 0 \, . $
Thus,
\beaa
  \derm_{\underline{L}}   \derm_{L}   (  \Lie_{Z^{K}}  \Phi_{U} ) &=&   \pa_{\underline{L}}  \Big(  \pa_{L}   (  \Lie_{Z^{K}}  \Phi_{U} ) \Big)   \; .
  \eeaa
As a result, we could write,
\bea
\notag
  \derm_{\underline{L}}   \derm_{L}   (  \Lie_{Z^{K}}  \Phi_{U} ) &=&    \pa_{L}  \Big(  \pa_{\underline{L}}  (  \Lie_{Z^{K}}  \Phi_{U} ) \Big)   =    \pa_{L}  \Big(  \derm_{\underline{L}}  (  \Lie_{Z^{K}}  \Phi_{U} )  \Big)   \\
  &=&    \derm_{L}    \derm_{\underline{L}}  (  \Lie_{Z^{K}}  \Phi_{U} )     \; .
  \eea
Hence,
 \bea
 \notag
 && \sum_{  |J| + |K| \leq |I|, \; |K| < |I| }  | ( \Lie_{Z^{J}}   H)_{L  \underline{L}}    \cdot  \derm_{\underline{L}}   \derm_{L}   (  \Lie_{Z^{K}}  \Phi_{U} )  |  \\
  \notag
 &\les&  \sum_{  |J| + |K| \leq |I|, \; |K| < |I| }  | ( \Lie_{Z^{J}}   H)_{L  \underline{L}} |   \cdot | \derm_{L}    \derm_{\underline{L}}  (  \Lie_{Z^{K}}  \Phi_{U} )  | \; .\\
 \eea
 
\textbf{The term $\derm_{\underline{L}}   \derm_{e_A}   (  \Lie_{Z^{K}}  \Phi_{U} )$}\,:\\

We look at the region $t \geq 1 $ or $r \geq 1$ and we distinguish the cases $q > 0$ and $q \leq 0$\,:\\
  
\textbf{In the case where $q = r-t > 0$}\,:

We showed in Lemma \ref{restrictedderivativesintermsofZ}, that $ \rpa_i  =   \frac{x^j}{r^2}   Z_{ij} \, ,$ and hence, we can write
 \bea
e_A = \frac{1}{r}C^{ij}_A Z_{ij} \, ,
\eea
where $C^{ij}_A$ are in fact linear combinations of  $\frac{x^j}{r}$. However, $ \pa_{\underline{L}} \frac{x^j}{r} = -  \frac{x^{i}}{r} ( \frac{\de_{i}^{\;\; j} }{r}  ) +   \frac{x^{i}}{r}   \frac{x^j}{r^2} ( \pa_{i} r  ) \; .$
We have $ \pa_{i} r = \frac{x_i}{r} \; .$ Thus,
\beaa
 \pa_{r} ( \frac{x^j}{r} ) &=&  \frac{x^{i}}{r} ( \frac{\de_{i}^{\;\; j} }{r}  ) -    \frac{x^{i}}{r}   \frac{x^j}{r^2} ( \pa_{i} r  )  =  \frac{x^{i}}{r} ( \frac{\de_{i}^{\;\; j} }{r}  ) -  \frac{x^{i} \cdot x_i}{r^2}   \frac{x^j}{r^2}  =  0 \;.
\eeaa
As a result
\bea\label{Lbarderivativeofxioverrisinfactzero}
\pa_{\underline{L}} \frac{x^j}{r} = 0 \; ,
\eea
and therefore, $ \pa_{\underline{L}} C^{ij}_A = 0 \, .$ We also have $\pa_{\underline{L}} \frac{1}{r} = - \pa_{r}  ( \frac{1}{r}  ) =   \frac{1}{r^2}   \, .$ Thus,
\bea
\pa_{\underline{L}} ( \frac{1}{r}  C^{ij}_A ) =  \frac{1}{r^2}   C^{ij}_A \; .
\eea
Since $  \derm_{e_A}   (  \Lie_{Z^{K}}  \Phi_{U} ) =  \frac{1}{r} C^{ij}_A \derm_{Z_{ij} }  (  \Lie_{Z^{K}}  \Phi_{U} ) \; , $ therefore,
  \beaa
  \notag
  \derm_{\underline{L}}   \derm_{e_A}   (  \Lie_{Z^{K}}  \Phi_{U} ) &=&  \pa_{\underline{L}} ( \frac{1}{r}  C^{ij}_A ) \cdot \derm_{Z_{ij} }  (  \Lie_{Z^{K}}  \Phi_{U} )  + \frac{1}{r} C^{ij}_A   \derm_{\underline{L}}  \derm_{Z_{ij} }  (  \Lie_{Z^{K}}  \Phi_{U} ) \\
    \notag
   &=&   \frac{1}{r^2}  C^{ij}_A \cdot \derm_{Z_{ij} }  (  \Lie_{Z^{K}}  \Phi_{U} )  + \frac{1}{r} C^{ij}_A   \derm_{\underline{L}}  \derm_{Z_{ij} }  (  \Lie_{Z^{K}}  \Phi_{U} ) \; . 
  \eeaa
However, $ Z_{ij} = x_{j} \pa_{i} - x_{i} \pa_{j} $\,, and thus,
\beaa
  \derm_{Z_{ij} }  (  \Lie_{Z^{K}}  \Phi ) = x_{j} \derm_{ \frac{\pa}{\pa x_i}}  (  \Lie_{Z^{K}}  \Phi )   - x_{i} \derm_{ \frac{\pa}{\pa x_j}}    (  \Lie_{Z^{K}}  \Phi )   \, .
\eeaa
Therefore,
\beaa
 | \frac{1}{r^2}  C^{ij}_A \cdot \derm_{Z_{ij} }  (  \Lie_{Z^{K}}  \Phi_{U} )  | &\leq&   | \frac{x_{j}}{r^2}  C^{ij}_A  |  \cdot | \derm_{ \frac{\pa}{\pa x_i}}  (  \Lie_{Z^{K}}  \Phi )| +    | \frac{x_{i}}{r^2}  C^{ij}_A  | \cdot | \derm_{ \frac{\pa}{\pa x_j}}    (  \Lie_{Z^{K}}  \Phi )  | \; .
 \eeaa
 Given that $C^{ij}_A$ and $\frac{x_{i}}{r}$ are bounded, we get 
 \bea
 \notag
 | \frac{1}{r^2}  C^{ij}_A \cdot \derm_{Z_{ij} }  (  \Lie_{Z^{K}}  \Phi_{U} )  | &\leq&    \frac{1}{r}   \cdot | \derm_{ \frac{\pa}{\pa x_i}}  (  \Lie_{Z^{K}}  \Phi )| +   \frac{1}{r} \cdot | \derm_{ \frac{\pa}{\pa x_j}}    (  \Lie_{Z^{K}}  \Phi )  | \\
 &\leq&    \frac{1}{r}   \cdot | \derm (  \Lie_{Z^{K}}  \Phi )|  \; .
 \eea
 Whereas to the term $\frac{1}{r} C^{ij}_A   \derm_{\underline{L}}  \derm_{Z_{ij} }  (  \Lie_{Z^{K}}  \Phi_{U} ) $\,, we first note that since it is a tensor, we have
 \beaa
 | \frac{1}{r} C^{ij}_A   \derm_{\underline{L}}  \derm_{Z_{ij} }  (  \Lie_{Z^{K}}  \Phi_{U} ) | \les  | \frac{1}{r} C^{ij}_A   \derm_{\underline{L}}  \derm_{Z_{ij} }  (  \Lie_{Z^{K}}  \Phi ) | \; .
 \eeaa

 Thus, we can compute, in our fixed system of coordinates, the right hand side of the above inequality in order to make an estimate. First, we compute for $\mu \in \{0, 1, 2, 3 \}$\,, 
 \beaa
   \derm_{Z_{ij} }  (  \Lie_{Z^{K}}  \Phi_{\mu} )      &=&   x_{j} \derm_{ \frac{\pa}{\pa x_i}}  (  \Lie_{Z^{K}}  \Phi_{\mu} )   - x_{i} \derm_{ \frac{\pa}{\pa x_j}}    (  \Lie_{Z^{K}}  \Phi_{\mu} )   \\
  &=&   x_{j} \pa_{i}  (  \Lie_{Z^{K}}  \Phi_{\mu} )   - x_{i} \pa_{j}    (  \Lie_{Z^{I_2}}  \Phi_{\mu} )   =  \pa_{Z_{ij}}  (  \Lie_{Z^{K}}  \Phi_{\mu} )   \\ 
    &=&  \Lie_{Z_{ij}}  (   \Lie_{Z^{K}}  \Phi_{\mu} )  +   \Lie_{Z^{K}}  \Phi ( [Z_{ij},  \frac{\pa}{\pa x_{\mu}} ] ) \\
     &=&  \Lie_{Z_{ij}}  \Lie_{Z^{K}}  \Phi_{\mu}  + \sum_{m} \Lie_{Z^{J_{m_1}} } \Lie_{ [Z_{ij}, Z^{m}] }  \Lie_{Z^{J_{m_2}} } \Phi_{\mu}   +  \Lie_{Z^{K}}  \Phi ( [Z_{ij},  \frac{\pa}{\pa x_{\mu}} ] ) \; ,
  \eeaa
  where by $Z^{m}$, we mean the $m$-th Minkowski vector field in the product $Z^K$, and the rest in the product is $Z^{J_{m_1}}$ (what is before $Z^m$) and $Z^{J_{m_2}}$ (what is after $Z^m$).

  However,  the commutation of two vector fields in $\cal Z$ is a linear combination of vector fields in $\cal Z$, and the commutation of a vector field in $\cal Z$ and of a vector $\pa_\mu$\,, $\mu \in  \{t, x^1, x^2, x^3 \}$, gives a linear combination of vectors of the form $\pa_\mu$. Using that fact, we get that for all $ \mu \in (t, x^1, x^2, x^3)$,
  \beaa
\derm_{Z_{ij} }  (  \Lie_{Z^{K}}  \Phi_{\mu} ) &=&  \sum_{i} a_i \Lie_{Z^{I_i}} \Phi_{\mu} + \sum_{i} b_i \cdot \Lie_{Z^{K}} \Phi_{\mu^i}  \; ,
\eeaa
for some $\mu^i \in  \{t, x^1, x^2, x^3 \}$\,, and $a_i\,,\, b_i$ are constants, and $|I_i| \leq |K| + 1$\,. Thus,
\beaa
 \derm_{\underline{L}}   \derm_{Z_{ij} }  (  \Lie_{Z^{K}}  \Phi_{\mu} ) &=&  \sum_{i} a_i    \derm_{\underline{L}}   ( \Lie_{Z^{I_i}} \Phi_{\mu})  + \sum_{i} b_i \cdot   \derm_{\underline{L}}   ( \Lie_{Z^{K}} \Phi_{\mu^i} ) 
\eeaa
and therefore, for all $\mu \in  \{t, x^1, x^2, x^3 \}$\,,
 \beaa
 |   \derm_{\underline{L}}  \derm_{Z_{ij} }  (  \Lie_{Z^{K}}  \Phi_{\mu} ) | \les \sum_{|M| \leq |K|+1} | \derm_{\underline{L}}   ( \Lie_{Z^{M}} \Phi )  | \; ,
 \eeaa
and as a result
 \beaa
 \notag
 | \frac{1}{r} C^{ij}_A   \derm_{\underline{L}}  \derm_{Z_{ij} }  (  \Lie_{Z^{K}}  \Phi_{U} ) |   \les \sum_{|M| \leq |K|+1}  \frac{1}{r} \cdot  | \derm  ( \Lie_{Z^{M}} \Phi )  | \; .
 \eeaa
 Finally, we obtain for $q > 0$\,,
   \bea
  \notag
  | \derm_{\underline{L}}   \derm_{e_A}   (  \Lie_{Z^{K}}  \Phi_{U} ) | &\les&  | \frac{1}{r^2}  C^{ij}_A \cdot \derm_{Z_{ij} }  (  \Lie_{Z^{K}}  \Phi_{U} )  |  +  | \frac{1}{r} C^{ij}_A   \derm_{\underline{L}}  \derm_{Z_{ij} }  (  \Lie_{Z^{K}}  \Phi_{U} ) |  \\
  &\les&    \sum_{|M| \leq |K|+1} \frac{1}{r}   \cdot | \derm (  \Lie_{Z^{M}}  \Phi )| \; .
  \eea
Since we look in the region $t \geq 1 $ or $r \geq 1$ and we are in the case where $q  > 0$\,, this imposes that $r \geq 1 $ and therefore, in that region, we have $\frac{1}{r} \les \frac{1}{1+r} \; , $ and as we have shown, for $q \geq 0$\,, we have $ \frac{1}{1+r} =  \frac{1}{1 + t + |q|  }  \;.$ Consequently,
   \beaa
  \notag
  | \derm_{\underline{L}}   \derm_{e_A}   (  \Lie_{Z^{K}}  \Phi_{U} ) | &\les&    \sum_{|M| \leq |K|+1}  \frac{1}{1+r}   \cdot | \derm (  \Lie_{Z^{M}}  \Phi )|  \\
  \notag
  &\les&  \frac{1}{(1 + t + |q|)  }  \cdot   \sum_{|M| \leq |K|+1} | \derm (  \Lie_{Z^{M}}  \Phi )| \; . 
  \eeaa
  Thus, in the region $t \geq 1 $ or $r \geq 1$\,, for $q  > 0$\,, we have
     \bea
  \notag
  && \sum_{  |J| + |K| \leq |I|, \; |K| < |I| }   |  ( \Lie_{Z^{J}}   H)_{L e_A }    \cdot  \derm_{\underline{L}}   \derm_{e_A}    (  \Lie_{Z^{K}}  \Phi_{U} )  | \\
    \notag
    &\les&    \sum_{  |J| + |K| \leq |I|, \; |K| < |I| }    \frac{1}{(1 + t + |q|)  }  \cdot   |  ( \Lie_{Z^{J}}   H)_{L e_A }   | \cdot | \sum_{|M| \leq |K|+1} | \derm (  \Lie_{Z^{M}}  \Phi )| \; .\\
  \eea
  
\textbf{ In the case where $q = r-t < 0$}\,:
 
 We then recall that we showed in Lemma \ref{restrictedderivativesintermsofZ}, that we have also a different presentation for $\rpa_{i}$\,, which is more suitable to make estimates in the region $q  < 0$\,, that is $
\rpa_{i} = \frac{  - \frac{x_i}{r} \frac{x^j}{r}   Z_{0j}  +  Z_{0i}    }{  t  }  \; , $ and therefore, we can see $e_A$ as linear combinations of the form 
 \bea
e_A = \frac{1}{t} G_{A}^{j} Z_{0j} +  \frac{1}{t} B_{A}^{i} Z_{0i} \, ,
\eea
where $G_{A}^{j}$ is a linear combination of $\frac{x_i}{r} \frac{x^j}{r}$ and $B_{A}^{i}$ are constants. Based on what we have shown in \eqref{Lbarderivativeofxioverrisinfactzero}, we have then $
\pa_{\underline{L}} G_{A}^{j}  = \pa_{\underline{L}} B_{A}^{j}  =  0 \, . $ Thus, we can write
 \bea
e_A = \frac{1}{t} C_{A}^{j} Z_{0j} \, ,
\eea
with $C_{A}^{j} $ bounded and furthermore, $\pa_{\underline{L}} C_{A}^{j}  = 0 \, .$ We also have $\pa_{\underline{L}} \frac{1}{t} =  \pa_{t}  ( \frac{1}{t}  ) =   - \frac{1}{t^2}   \, .$ Thus,
\beaa
\pa_{\underline{L}} ( \frac{1}{t}   C_{A}^{j}   ) &=&  - \frac{1}{t^2}   C^{j}_A \; .
\eeaa
Since
\beaa 
  \derm_{e_A}   (  \Lie_{Z^{K}}  \Phi_{U} ) &=& \frac{1}{t} C_{A}^{j} \derm_{Z_{0j} }  (  \Lie_{Z^{K}}  \Phi_{U} )   \; ,
\eeaa
therefore,
  \bea\label{Lbareacovarderiderivativeofliederivativephi}
  \notag
  \derm_{\underline{L}}   \derm_{e_A}   (  \Lie_{Z^{K}}  \Phi_{U} ) &=&  \pa_{\underline{L}} ( \frac{1}{t} C_{A}^{j}  ) \cdot \derm_{Z_{0j}}  (  \Lie_{Z^{K}}  \Phi_{U} )  + \frac{1}{t} C_{A}^{j}   \derm_{\underline{L}}  \derm_{Z_{0j} }  (  \Lie_{Z^{K}}  \Phi_{U} ) \\
    \notag
   &=& -  \frac{1}{t^2} C_{A}^{j}  \cdot \derm_{Z_{0j} }  (  \Lie_{Z^{K}}  \Phi_{U} )  + \frac{1}{t} C_{A}^{j}    \derm_{\underline{L}}  \derm_{Z_{0j} }  (  \Lie_{Z^{K}}  \Phi_{U} ) \; . 
  \eea
However, $Z_{0j} =   x_{j} \pa_{t} + t \pa_{j}  $\,, and thus,
\beaa
  \derm_{Z_{0j} }  (  \Lie_{Z^{K}}  \Phi ) = x_{j} \derm_{ \frac{\pa}{\pa t}}  (  \Lie_{Z^{K}}  \Phi )   + t  \derm_{ \frac{\pa}{\pa x_j}}    (  \Lie_{Z^{K}}  \Phi )   \, .
\eeaa
Therefore, on one hand
\beaa
 | \frac{1}{t^2}  C^{j}_A \cdot \derm_{Z_{0j} }  (  \Lie_{Z^{K}}  \Phi_{U} )  | &\leq&   | \frac{x_{j}}{t^2}  C^{j}_A  |  \cdot | \derm_{ \frac{\pa}{\pa t}}  (  \Lie_{Z^{K}}  \Phi )| +    | \frac{t}{t^2}  C^{j}_A  | \cdot | \derm_{ \frac{\pa}{\pa x_j}}    (  \Lie_{Z^{K}}  \Phi )  | \; .
 \eeaa
 Now, we recall that we are in the region $q = r-t < 0$ and therefore in that region, $ \frac{r}{t} \leq 1 \, ,$ and therefore, since also $C^{j}_A$ is bounded, we get 
 \bea
 \notag
 | \frac{1}{t^2}  C^{j}_A \cdot \derm_{Z_{0j} }  (  \Lie_{Z^{K}}  \Phi_{U} )  |  &\leq&    \frac{1}{t}   \cdot | \derm_{ \frac{\pa}{\pa t}}  (  \Lie_{Z^{K}}  \Phi )| +   \frac{1}{t} \cdot | \derm_{ \frac{\pa}{\pa x_j}}    (  \Lie_{Z^{K}}  \Phi )  | \\
 &\leq&    \frac{1}{t}   \cdot | \derm (  \Lie_{Z^{K}}  \Phi )|  \; .
 \eea
On the other hand, for the other term in \eqref{Lbareacovarderiderivativeofliederivativephi}, proceeding as earlier, we obtain
 \beaa
 | \frac{1}{t} C_{A}^{j}    \derm_{\underline{L}}  \derm_{Z_{0j} }  (  \Lie_{Z^{K}}  \Phi_{U} ) |&\les&   | \frac{1}{t} C_{A}^{j}    \derm_{\underline{L}}  \derm_{Z_{0j} }  (  \Lie_{Z^{K}}  \Phi) | \\
&\les& \sum_{|M| \leq |K|+1}  \frac{1}{t} \cdot  | \derm  ( \Lie_{Z^{M}} \Phi )  | \; .
 \eeaa
 Finally, we obtain for $q < 0$\,,
   \bea
  \notag
  |   \derm_{\underline{L}}   \derm_{e_A}   (  \Lie_{Z^{K}}  \Phi_{U} ) |  &\les&    \sum_{|M| \leq |K|+1} \frac{1}{t}   \cdot | \derm (  \Lie_{Z^{M}}  \Phi )| \; .
  \eea
Since we look in the region $t \geq 1 $ or $r \geq 1$ and we are in the case where $q  < 0$\,, this imposes that $ \frac{1}{t} \les  \frac{1}{1+t} \les  \frac{1}{1 + t + |q|  } \; ,$ and hence
   \bea
  \notag
  | \derm_{\underline{L}}   \derm_{e_A}   (  \Lie_{Z^{K}}  \Phi_{U} ) |  &\les&  \frac{1}{(1 + t + |q|)  }  \cdot   \sum_{|M| \leq |K|+1} | \derm (  \Lie_{Z^{M}}  \Phi )| \; . 
  \eea
  Thus, in the region $t \geq 1 $ or $r \geq 1$\,, for $q  < 0$\,, we have
     \bea
  \notag
  && \sum_{  |J| + |K| \leq |I|, \; |K| < |I| }   |  ( \Lie_{Z^{J}}   H)_{L e_A }    \cdot  \derm_{\underline{L}}   \derm_{e_A}    (  \Lie_{Z^{K}}  \Phi_{U} )  | \\
    \notag
    &\les&    \sum_{  |J| + |K| \leq |I|, \; |K| < |I| }    \frac{1}{(1 + t + |q|)  }  \cdot   |  ( \Lie_{Z^{J}}   H)_{L e_A }   | \cdot | \sum_{|M| \leq |K|+1} | \derm (  \Lie_{Z^{M}}  \Phi )| \; .
  \eea
Finally, putting all together and noting that the calculations are also true for a two-tensor, we obtain the result.
 \end{proof}

\begin{lemma}\label{betterdecayfortangentialderivatives}
We have the following well-known inequality for all $t \geq 0$ and for all $q \in \R$\,,
\beaa
(1 + t + |q| ) \cdot |\rderm \Lie_{Z^I}  \Phi | \les \sum_{|J| \leq |I| + 1} | \Lie_{Z^J}  \Phi | \, .
\eeaa
\end{lemma}

Now, we are ready to prove our decoupled estimate on the commutator term.
\begin{lemma}
For any $V \in \cal T$, we have the decoupled commutator estimate in \eqref{Theseperatecommutatortermestimateforthetangentialcomponents}.  
  \end{lemma}
 \begin{proof}
In the result of Lemma \ref{commutationformaulamoreprecisetoconservegpodcomponentsstructure}, we estimate
  \beaa
  &&    | ( \Lie_{Z^{J}}   H)_{L  \underline{L}} |   \cdot | \derm_{L}    \derm_{\underline{L}}  (  \Lie_{Z^{K}}  \Phi_{U} )  | \\
\notag
 && +       | m^{\mu \b}   ( \Lie_{Z^{J}}   H)_{\underline{L}\b}  \cdot    \derm_{L}   \derm_{\mu}     (  \Lie_{Z^{K}}  \Phi_{U} )  | + |   m^{\mu \b}   ( \Lie_{Z^{J}}   H)_{e_A \b}   \cdot   \derm_{e_A}   \derm_{\mu}    (  \Lie_{Z^{K}}  \Phi_{U} )   | \Big) \\
 \notag
 &\les &   |  \Lie_{Z^{J}}   H |   \cdot | \derm_{L}    \derm  (  \Lie_{Z^{K}}  \Phi )  |  + |  \Lie_{Z^{J}}   H |  \cdot  |  \derm_{L}   \derm     (  \Lie_{Z^{K}}  \Phi )  | \\
 \notag
 && + | \Lie_{Z^{J}}   H |     \cdot  | \derm_{e_A}   \derm    (  \Lie_{Z^{K}}  \Phi )   | \Big) \\
 &\les & | \Lie_{Z^{J}}   H |     \cdot  | \rderm   \derm    (  \Lie_{Z^{K}}  \Phi )   | \,.
  \eeaa
  In our fixed system of coordinates, for a given $\mu\,, \nu \in \{t, x^1, x^2, x^3 \}$\,, we can look at each component and estimate, as we have shown in Lemma \ref{betterdecayfortangentialderivatives}, that
  \beaa
   | \rderm   \derm_{\mu}    (  \Lie_{Z^{K}}  \Phi_{\nu} )   | &\leq&    | \rpa   \derm_{\mu}    (  \Lie_{Z^{K}}  \Phi_{\nu} )   | \\
   &\leq& \frac{1}{(1 + t + |q| )}  \cdot \sum_{|I| = 1} |Z^I  \derm_{\mu}    (  \Lie_{Z^{K}}  \Phi_{\nu} )   | \, .
\eeaa
Using the fact that commutation of a Minkowski vector field $Z$ and a coordinate vector field $\frac{\pa}{\pa x_{\mu}}$\,, gives a linear combination of coordinates vector fields, we get
\beaa
 \sum_{|I| = 1}  |Z^I  \derm_{\mu}    (  \Lie_{Z^{K}}  \Phi_{\nu} )   | \les \sum_{|I| \leq 1}  |\Lie_{Z^I}   \derm (  \Lie_{Z^{K}}  \Phi )   |  \; .
\eeaa
Summing over all indices $\mu\,, \nu \in \{t, x^1, x^2, x^3 \}$\,, we obtain
\beaa
  | \rderm   \derm    (  \Lie_{Z^{K}}  \Phi )   |  \les \frac{1}{(1 + t + |q| )}  \cdot  \sum_{|I| \leq 1}  |\Lie_{Z^I}   \derm (  \Lie_{Z^{K}}  \Phi )   |  \;.
\eeaa
As a result, using the fact that the Lie derivative $\Lie_{Z^{I}}$ commutes with $\derm$, we get
 \beaa
  &&    | ( \Lie_{Z^{J}}   H)_{L  \underline{L}} |   \cdot | \derm_{L}    \derm_{\underline{L}}  (  \Lie_{Z^{K}}  \Phi_{U} )  | \\
\notag
 && +       | m^{\mu \b}   ( \Lie_{Z^{J}}   H)_{\underline{L}\b}  \cdot    \derm_{L}   \derm_{\mu}     (  \Lie_{Z^{K}}  \Phi_{U} )  | + |   m^{\mu \b}   ( \Lie_{Z^{J}}   H)_{e_A \b}   \cdot   \derm_{e_A}   \derm_{\mu}    (  \Lie_{Z^{K}}  \Phi_{U} )   | \Big) \\
 \notag
 &\les & | \Lie_{Z^{J}}   H |     \cdot  \frac{1}{(1 + t + |q| )}  \cdot   \sum_{|M| \leq |K| +1}  |   \derm (  \Lie_{Z^{M}}  \Phi )   | \; .
 \eeaa
Consequently, for any $V \in {\cal T}$,
 \beaa
\notag
&&| \Lie_{Z^I}  ( g^{\la\mu} \derm_{\la}   \derm_{\mu}     \Phi_{V} ) - g^{\la\mu}    \derm_{\la}   \derm_{\mu}  (  \Lie_{Z^I} \Phi_{V}  ) |  \\
\notag
&\les&  \sum_{|K| < |I| }  | g^{\la\mu} \cdot \derm_{\la}   \derm_{\mu} (  \Lie_{Z^{K}}  \Phi_{V} ) | \\
\notag
&& +  \sum_{|J| + |K| \leq |I|, \; |K| < |I| }  \Big(   |   ( \Lie_{Z^{J}}   H)_{L  L} |   \cdot    \frac{1}{(1+t+|q|)} \cdot  \sum_{|M| \leq |K|+1}  | \derm ( \Lie_{Z^M}  \Phi )  |  \\
   \notag
 && +    |   ( \Lie_{Z^{J}}   H)_{L  L}   | \cdot   \frac{1}{(1+|q|)} \cdot   \sum_{|M| \leq |K|+1}    \sum_{ V^\prime \in {\cal T} } | \derm   ( \Lie_{Z^M}  \Phi _{V^\prime} ) |   \\
\notag
    &&  +    \frac{1}{(1 + t + |q|)  }  \cdot   |  ( \Lie_{Z^{J}}   H)_{L e_A }   | \cdot | \sum_{|M| \leq |K|+1} | \derm (  \Lie_{Z^{M}}  \Phi )| \\
 && +  | \Lie_{Z^{J}}   H |     \cdot  \frac{1}{(1 + t + |q| )}  \cdot \sum_{|M| \leq |K| +1}  |   \derm (  \Lie_{Z^{M}}  \Phi )   |  \\
 &\les&  \sum_{|K| < |I| }  | g^{\la\mu} \cdot \derm_{\la}   \derm_{\mu} (  \Lie_{Z^{K}}  \Phi_{V} ) | \\
\notag
&& +  \sum_{|J| + |K| \leq |I|, \; |K| < |I| }  \Big(   |   \Lie_{Z^{J}}   H  |   \cdot    \frac{1}{(1+t+|q|)} \cdot  \sum_{|M| \leq |K|+1}  | \derm ( \Lie_{Z^M}  \Phi )  |  \\
   \notag
 && +    |   ( \Lie_{Z^{J}}   H)_{L  L}   | \cdot   \frac{1}{(1+|q|)} \cdot   \sum_{|M| \leq |K|+1}    \sum_{ V^\prime \in {\cal T} } | \derm   ( \Lie_{Z^M}  \Phi _{V^\prime} ) | \Big)  \; .
  \eeaa
Thus, we get the result.
\end{proof}

\textbf{Conflict of Interest Statement:} The work in this manuscript was funded by the Beijing Institute of Mathematical Sciences and Applications (BIMSA) in China.

\textbf{Data Availability Statement:} No datasets were generated or analysed during the current study in this manuscript.

\end{document}